\documentclass{amsart}
\usepackage{amssymb}
\usepackage{graphicx}
\usepackage{amsmath}
\usepackage{amsfonts}
\usepackage{subfigure}
\usepackage{hyperref}
\usepackage{todonotes}
\usepackage[T1]{fontenc}

\newcommand{\complex}{\mathbb{C}}

\newcommand{\disk}{{\mathbb D}}
\newcommand{\ucirc}{ {\mathbb S}}
\newcommand{\real}{{\mathbb R}}
\newcommand{\zed}{{\mathbb Z}}
\newcommand{\0}{\emptyset}
\newcommand{\lam}{{\mathcal L}}

\newcommand{\edis}{\operatorname{d_E}}
\newcommand{\sm}{\setminus}

\newcommand{\cl}{\overline}
\newcommand{\comb}[2]{\left( \begin{array}{c} #1\\ #2 \end{array}\right)}

\newtheorem{thm}{Theorem}[section]
\newtheorem{lem}[thm]{Lemma}
\newtheorem{cor}[thm]{Corollary}
\newtheorem{prop}[thm]{Proposition}

\theoremstyle{definition}
\newtheorem{ques}{Question}
\newtheorem{defn}[thm]{Definition}
\newtheorem{mydef}[thm]{Definition}

\newtheorem{ex}[thm]{Example}

\theoremstyle{remark}
\newtheorem{rem}[thm]{Remark}
\numberwithin{equation}{section}

\begin{document}

\title[Central Strips of Sibling Leaves]{Central Strips of Sibling Leaves in Laminations of the Unit Disk}

\author[D.~Cosper]{David J. Cosper}
\email[David J. Cosper]{dcosper314@gmail.com}

\author[J.~Houghton]{Jeffrey K. Houghton}
\email[Jeffrey K. Houghton]{echo41011@gmail.com }

\author[J.~Mayer]{John C. Mayer}
\email[John C.~Mayer]{jcmayer@uab.edu}

\author[L.~Mernik]{Luka Mernik}
\email[Luka Mernik]{mernik.1@osu.edu}

\author[J.~Olson]{Joseph W. Olson}
\email[Joseph W. Olson]{joseph.wesley.olson@gmail.com}

\begin{abstract}

Quadratic laminations of the unit disk were introduced by Thurston as a vehicle for understanding the (connected) Julia sets of quadratic polynomials and the parameter space of quadratic polynomials.  The ``Central Strip Lemma'' plays a key role in Thurston's
classification of gaps in quadratic laminations, and in describing the corresponding parameter space.  We generalize the notion of {\em Central Strip} to laminations of all degrees $d\ge2$ and prove a Central Strip Lemma for degree $d\ge2$.  We conclude with applications of the Central Strip Lemma to {\em identity return polygons} that show for higher degree laminations it may play a  role similar to Thurston's lemma.

\end{abstract}



\date{\today}

\thanks{The  authors thank the UAB
Laminations seminar, and in particular Drs. Lex Oversteegen and
Alexander Blokh and graduate student Ross Ptacek for useful conversations.  The authors also thank the referee for useful comments that improved the presentation.}

\maketitle

\section{Introduction}

Quadratic laminations of the unit disk were introduced by Thurston as a vehicle for understanding the (connected) Julia sets of quadratic polynomials and the parameter space of quadratic polynomials.  The ``Central Strip Lemma'' plays a key role in Thurston's
classification of gaps in quadratic laminations \cite{Thurston:2009}.
It is used to show that there are no wandering polygons for the
angle-doubling map $\sigma_2$ on the unit circle.  Moreover, when a polygon returns to itself, the iteration of $\sigma_2$ is transitive on the vertices.  For $\sigma_2$ it is sufficient to prove these facts for triangles: there are no {\em wandering triangles}, and no {\em identity return triangles}. From these facts,
the classification of types of gaps of a quadratic lamination, and a
parameter space for quadratic laminations, follows. Thurston posed a
question in his notes on laminations that he deemed important to
further progress in the field: ``Can there be wandering triangles
for $\sigma_3$ (and higher degree)?''  When Blokh and Oversteegen
\cite{BlokhOversteegen:2004,BlokhOversteegen:2009}  showed that the answer was ``yes,'' the need to
define and understand central strips for higher degree and their role in
``controlling'' wandering and identity return polygons of a lamination became imperative.  Contributions to this understanding were made by Goldberg \cite{Goldberg:93}, Milnor  \cite{Milnor:2006fr, Milnor:2000}, Kiwi \cite{Kiwi:2002}, Blokh and Levin \cite{BlokhLevin:2002}, Childers \cite{Childers:2007},  and others.

New results in this paper include Definition~\ref{sib-portrait} of a {\em
sibling portrait}, Definition~\ref{def-centralstrip} of a {\em
central strip},  the statement and proof of the Central Strip Structure Theorem~\ref{lem-centralstrip}, the statement and proof of Theorem~\ref{thm-centralstriplemma} (the generalized Central Strip
Lemma), and Theorem~\ref{centraltree}
counting the number of different sibling portraits that could correspond to the preimage of 
a given non-degenerate leaf.  The Central Strip
Lemma can be used to provide new proofs of known results that the
authors believe to be more transparent,  yield more onformation about the laminations, and yield new results for
laminations of degree $d\ge 3$.  Initial applications to identity return triangles under $\sigma_3$ appear in this paper (see Section~\ref{app}), and other applications, particularly to higher degree, will appear in  subsequent papers.

\subsection{Preliminaries}

Let $\complex$ denote the complex plane, $\disk\subset\complex$ the
open unit disk, $\cl\disk$ the closed unit disk, and $\ucirc$ the
boundary of the unit disk (i.e., the unit circle), parameterized as
$\real/\zed$.  For $d\ge 2$, define a map $\sigma_d:\ucirc\to\ucirc$ by
$\sigma_d(t)=dt \pmod 1$.

 \begin{defn} A {\em lamination} $\lam$ is a collections of chords of
$\cl\disk$, which we call {\em leaves}, with the property that any
two leaves meet, if at all, in a point of $\ucirc$, and
such that $\lam$  has the property that
$$\lam^*:=\ucirc\cup\{\cup\lam\}$$ is a closed subset of
$\cl\disk$.  \end{defn}  

It follows that $\lam^*$ is a continuum (compact, connected metric space).  We allow {\em
degenerate} leaves -- all points of $\ucirc$ are degenerate leaves.
If  $\ell\in\lam$ is a leaf, we write $\ell=\cl{ab}$, where $a$ and
$b$ are the endpoints of $\ell$ in $\ucirc$.  We let
$\sigma_d(\ell)$ be the chord $\cl{\sigma_d(a)\sigma_d(b)}$.  If it
happens that $\sigma_d(a)=\sigma_d(b)$, then $\sigma_d(\ell)$ is a
point, called a {\em critical value} of $\lam$ and we say $\ell$ is
a {\em critical} leaf.

\begin{prop}
Let $\sigma_d^*$ denote the linear extension of $\sigma_d$ to leaves
of $\lam$, so that $\sigma_d^*$ is defined on $\lam^*$.  Then
$\sigma_d^*$ is continuous on $\lam^*$. 
\end{prop}

The proof is left to the reader.

\subsection{Leaf Length Function}  Fix the counterclockwise order $<$ on
$\ucirc$ as the preferred (circular) order.
Let $|(a,b)|$ denote the length in the parameterization of $\ucirc$
of the  arc in $\ucirc$ from $a$ to $b$ counterclockwise.
 Given a chord
$\cl{ab}$, there are two arcs of $\ucirc$ subtended by $\cl{ab}$.
 Define the {\em length} of $\cl{ab}$, denoted $|\cl{ab}|$,  to be
the shorter of $|(a,b)|$ or $|(b,a)|$.
 The maximum length of a leaf is thus $\frac12$.
    Note that the length of a critical leaf is $\frac id$
for some $i\in\{1,2, \dots j\mid j=\lfloor{\frac d2}\rfloor\}$.

 \begin{figure}[h] \label{leaflengthfn}
    \begin{center}
    \includegraphics[scale=0.4]{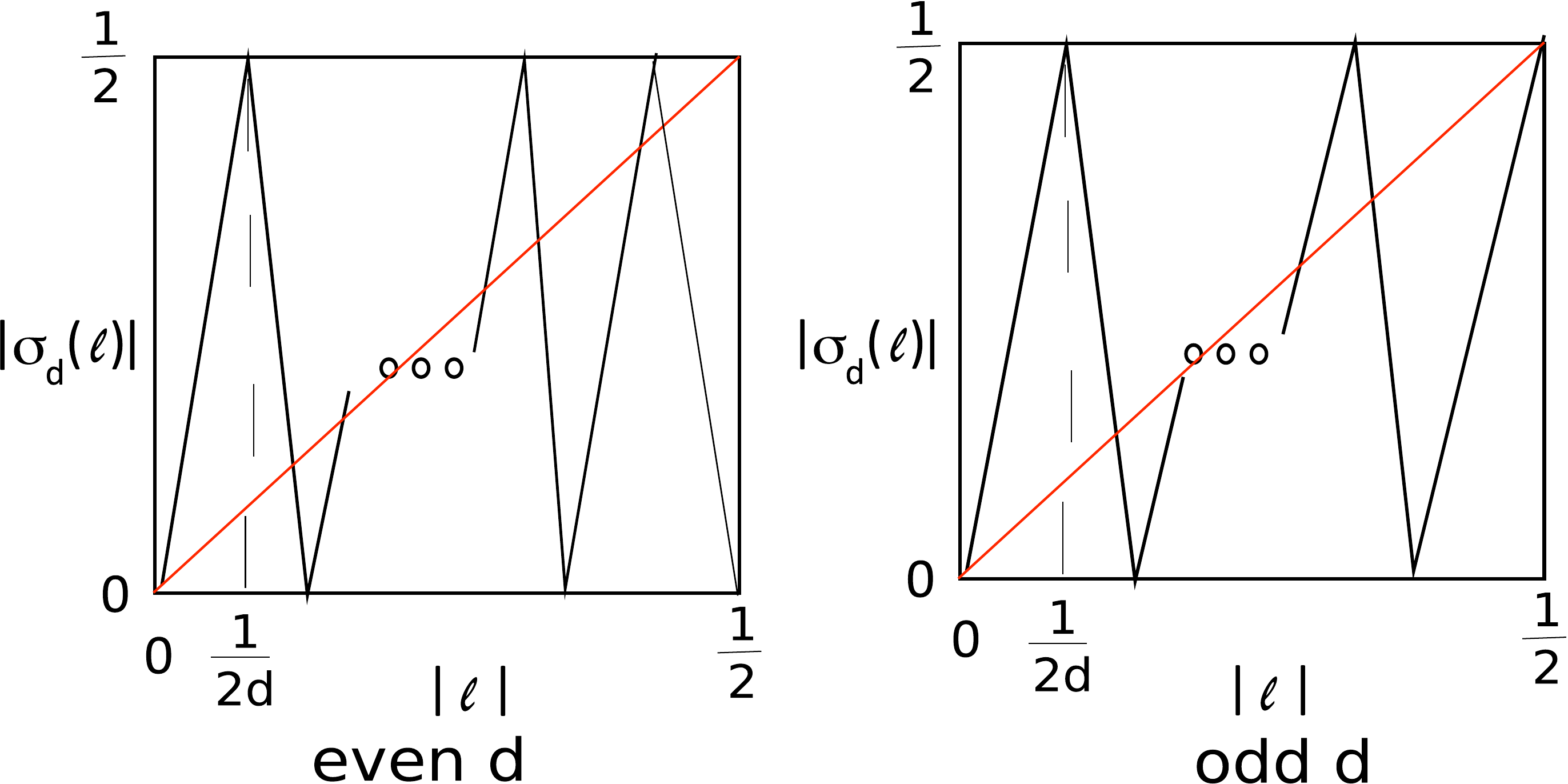}
    \caption{Graph of the leaf length function.}
    \end{center}
  \end{figure}

\begin{lem}\label{lem-tau} Let $x=|\ell|$ where $\ell\in\lam$.  Then
the length $\tau_d(x)$ of $\sigma_d(\ell)$ is given by the function

$$\tau_d(x)=\left\{\begin{array}{ccc}     dx $  if $ 0\leq x \le \frac1{2d}\\
                    1-dx $ if $ \frac1{2d}\leq x \leq \frac1d\\
                    dx-1 $ if $ \frac1d\le x \le \frac3{2d}\\
                    2-dx $ if $ \frac3{2d}\le x \le \frac2{d}\\
                    \vdots \\
                    (-1)^d dx + (-1)^{d+1} \lfloor \frac {d-1}2 \rfloor$ if $ \frac {d-2}{2d}\leq x \leq \frac{d-1}{2d}\\
                    (-1)^{d+1} dx + (-1)^{d+2} \lfloor \frac d2 \rfloor$ if $ \frac{d-1}{2d} \le x \le \frac12 \\
\end{array}\right.$$
defined on the interval $[0,1/2]$.
\end{lem}

The proof is left to the reader.

\begin{prop}\label{fixedlengths}
The fixed points of $\tau_d$ are of the form $$\{0,\frac{1}{d+1},\frac{1}{d-1},\frac{2}{d+1},\frac{2}{d-1},\dots, \frac{j}{d+1},\frac{j}{d-1}, \frac{j+1}{d+1} \dots, \le\frac12\}.$$  Thus, if $\frac{j}{d+1}<|\ell|<\frac{j}{d-1}$ for some $j$, then $|\sigma_d(\ell)|<|\ell|$, and if $\frac{j}{d-1}<|\ell|<\frac{j+1}{d+1}$ for some $j$, then $|\sigma_d(\ell)|>|\ell|$.
\end{prop}

The proof is left to the reader.

\begin{lem}\label{lem-length}
A leaf $\ell$ of length $|\ell|<\frac1{d+1}$ will keep increasing in
length under iteration of $\sigma_d$ until the length of some
iterate $|\sigma_d^i(\ell)|\ge \frac1{d+1}$.
\end{lem}

\begin{proof}
Consider $\tau_d$ and the identity function. Each monotone interval
 of the map $\tau_d$ and the identity function are linear. The first
interval of $\tau_d$, defined on $[0,1/2d]$, has the equation
$\tau_d(x)=dx$, therefore $(0,0)$ is the only intersection on the
first interval. On the second interval, $[1/2d, 1/d]$,
$\tau_d(x)=1-dx$. Therefore the intersection is $1-dx=x$ or
$x=\frac1{d+1}$. Since  the graph of $\tau_d$ is above the identity
function on interval $(0,\frac1{d+1})$, the length will keep
increasing under iteration until it is at least $\frac1{d+1}$.
\end{proof}

\subsection{Sibling Invariant Laminations}

Thurston's definition \cite{Thurston:2009} of invariant laminations did not involve
sibling leaves.  Blokh, Mimbs,  Oversteegen, and Valkenburg showed in \cite{Mimbs:2013} that each {\em sibling invariant lamination} (defined below) is a Thurston lamination, and each lamination induced by a locally
connected Julia set is a sibling invariant lamination.  They showed that to
understand Julia sets via laminations, it is sufficient to consider
sibling laminations.  (More precisely, they showed that the closure of the space of quadratic sibling laminations in the Hausdorff metric contains all laminations induced by locally connected Julia sets.)

\begin{mydef}\label{def-sibling}(Sibling Leaves)
Let $\ell_1\in\lam$ be a  leaf  and suppose
$\sigma_d(\ell_1)=\ell'$, for some non-degenerate leaf $\ell'\in\lam$.
A  leaf
$\ell_2\in\lam$, disjoint from $\ell_1$, is called a {\em sibling} of
$\ell_1$ provided $\sigma_d(\ell_2)=\ell'=\sigma_d(\ell_1)$.  A collection
${\mathcal S}=\{\ell_1, \ell_2, \dots, \ell_d\}\subset\lam$ is called a
{\em full sibling collection} provided that for each $i$,
$\sigma_d(\ell_i)=\ell'$ and for all $i\not=j$,
$\ell_i\cap\ell_j=\0$.
\end{mydef}

\begin{defn} \label{sibinvariant}
A lamination $\lam$ is said to be {\em sibling $d$-invariant} (or simply {\em invariant} if no confusion will result) provided that
\begin{enumerate}
\item (Forward Invariant) For every $\ell\in\lam$, $\sigma_d(\ell)\in\lam$.
\item (Backward Invariant) For every non-degenerate $\ell'\in\lam$,
there is a leaf $\ell\in\lam$ such that $\sigma_d(\ell)=\ell'$.
\item (Sibling Invariant) For every $\ell_1\in\lam$ with $\sigma_d(\ell_1)=\ell'$,
a non-degenerate leaf, there is a full sibling collection $\{\ell_1, \ell_2, \dots, \ell_d\}\subset\lam$ such that  $\sigma_d(\ell_i)=\ell'$.
\end{enumerate}
\end{defn}

\begin{defn} A {\em gap} in a lamination $\lam$ is the closure of a component of $\cl\disk\setminus\lam^*$.  A gap is {\em critical} iff two points in its boundary map to the same point.  A finite gap is usually called a {\em polygon}.  The leaves bounding a finite gap are called the {\em sides} of the polygon.  A polygon  is called {\em all-critical} if every side is a critical leaf.\end{defn}

\begin{defn}[Sibling Portrait]\label{sib-portrait}
The {\em sibling portrait} $S$ of a full collection of sibling
leaves is the collection of regions complementary to the sibling
leaves. We call a complementary region a {\em C-region} provided all
of the arcs in which the closure of the region meets the circle are
{\em short} (length $<\frac1{2d}$), and call it an {\em R-region} if
all of the arcs are {\em long} (length $>\frac1{2d}$).  
The {\em
degree} of a complementary region $T$, denoted $\text{deg}(T)$  is
number of leaves in the boundary of $T$ or, equivalently, the number
of circular arcs in the boundary of $T$. \end{defn}

$C$-regions which meet the circle in more than one short arc will constitute the components of the {\em central strip}, Definition~\ref{def-centralstrip}.  We show in Theorem~\ref{lem-centralstrip}, subject to the condition that no sibling maps to a diameter, that each region is either a $C$-region or an $R$-region.

Topologically, a {\em graph} is a finite union of arcs (homeomorphic images of the interval $[0,1]$) meeting only at endpoints.  Endpoints of these arcs are called {\em vertices} and the arcs themselves are called {\em edges}. The {\em degree} of a vertex $v$ is the number of edges that share $v$ as an endpoint.  A {\em tree} is a graph with no closed loops of edges in it.

\begin{defn}
The {\em dual graph} $T_{S}$ of the sibling portrait $S$ of a full collection $\mathcal S$ of sibling leaves is defined as follows: let each complementary region  correspond to a vertex of $T_{S}$ and each sibling leaf on the boundary of two regions correspond to an edge of $T_{S}$ between the vertices corresponding to the two regions.
\end{defn}

\begin{prop}
The dual graph of a sibling portrait under $\sigma_d$ is a connected tree consisting of $d+1$ vertices (components of the portrait), and $d$ edges (sibling leaves between components that meet on their boundaries).
\end{prop}

The proof is left to the reader.  See Figure~\ref{treepic} for an example. 

\begin{prop}
Let $S$ be the sibling portrait of a full collection $\mathcal S$ of sibling leaves under $\sigma_d$.  Let $T$ denote a complementary region of $S$ and $T'$ the corresponding vertex of the dual graph $T_{S}$.  Then $deg(T)=deg(T')$.
\end{prop}

The proof is left to the reader.

In Theorem~\ref{lem-centralstrip} below, we show that if the image leaf of a full sibling collection is not a diameter, then each of the
complementary regions of the sibling portrait is either a C-region or
an R-region of degree $d\ge 2$, or a {\em terminal} C- or R-region of
degree $1$. (By ``terminal'' region we mean a region corresponding to an endpoint of the dual graph.) Examples of sibling portraits for degrees $d=2$ and $d=3$
are in Figures~\ref{sibportpic2} and  \ref{sibportpic3}.
Figure~\ref{sibportpic} shows one of many possibilities for a
sibling portrait for $\sigma_6$.

 \begin{figure}[h]
    \begin{center}
   \includegraphics[scale=1]{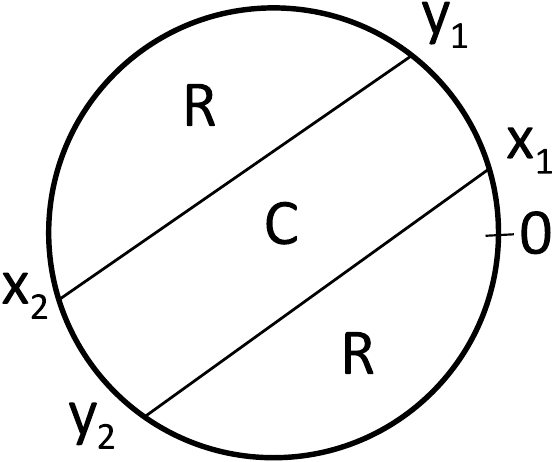}
    \caption{Example of a sibling portrait with a central strip for $\sigma_2$.}  \label{sibportpic2}
    \end{center}
  \end{figure}
  
  \begin{rem}
  Note that up to labeling of vertices and rotation, there are only two sibling portraits for $d=3$.  Moreover, the sibling leaves, $\ell_1$, $\ell_2$, and $\ell_3$ in the non-symmetric case are of three different lengths: $\frac13<|\ell_1|<\frac12$, $\frac16<|\ell_2|<\frac13$, and $|\ell_3|<\frac16$.  See Theorem~\ref{centraltree} for the count for $d>3$.
  \end{rem}

   \begin{figure}[h]
    \begin{center}
    \includegraphics[scale=0.9]{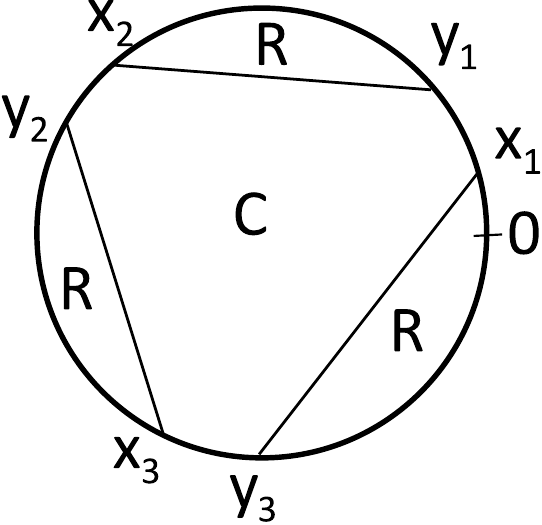}  \hspace{20pt}
     \includegraphics[scale=0.9]{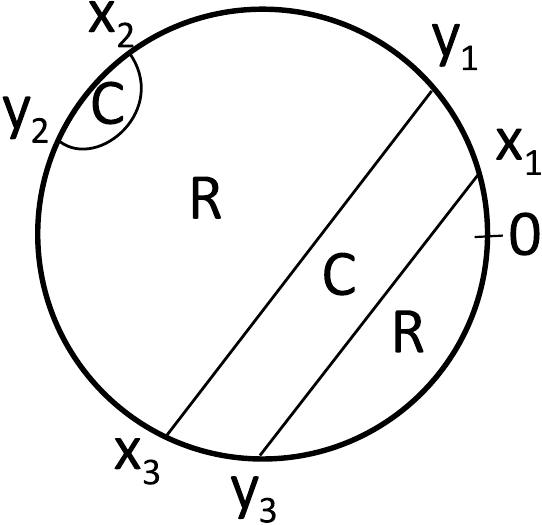}
    \caption{Examples of all sibling portraits with a central strip for $\sigma_3$
    with the same spacing of $x_i$ and $y_i$ points, up to rotational symmetry.}   \label{sibportpic3}
    \end{center}
  \end{figure}

 \begin{figure}[h]
    \begin{center}
    \includegraphics[scale=0.4]{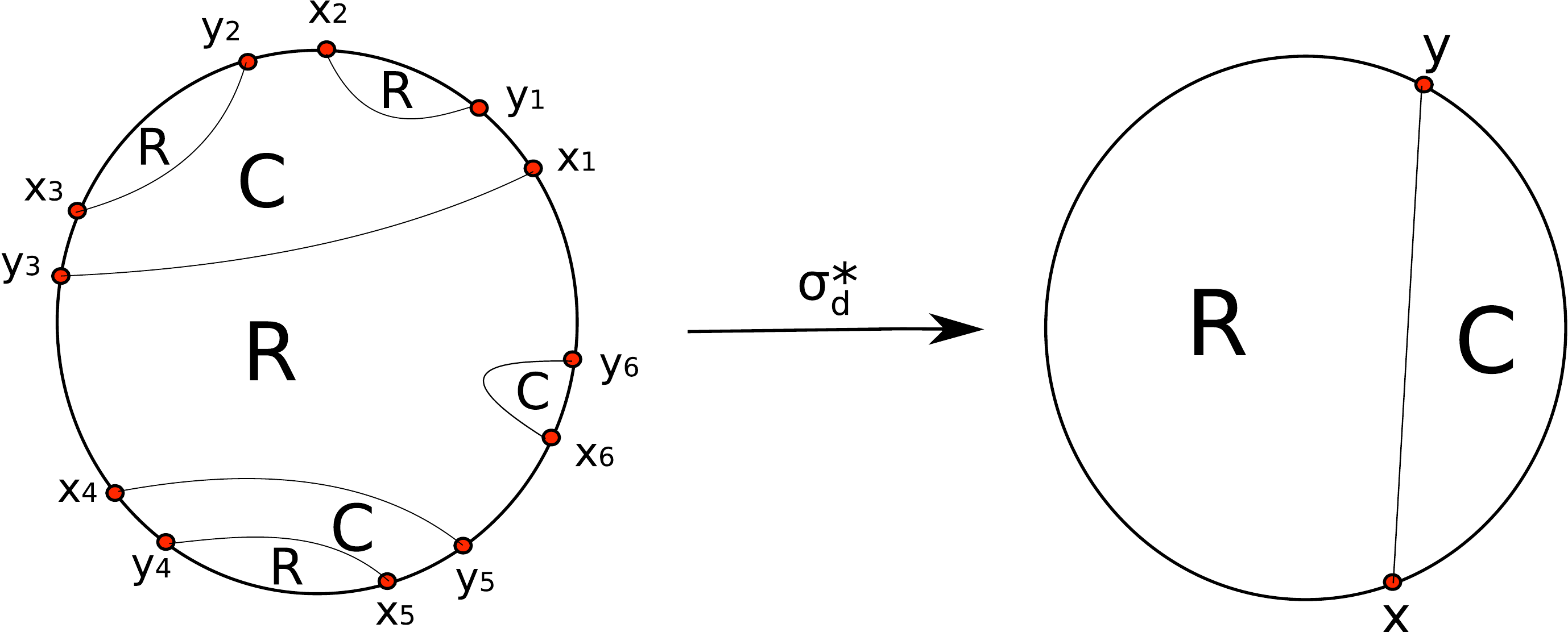}
    \caption{Example of a sibling portrait for $\sigma_6$. (Not to scale; though chords are straight, we sometimes draw them curved to stand out.)}    \label{sibportpic}
    \end{center}
  \end{figure}

\begin{prop}\label{degreeformula}
Suppose $S$ is a sibling portrait.  Then the following formula holds:
\[
\sum_{T\in S} (deg(T)-1)=d-1
\]
\end{prop}

The proof is a direct consequence of the Euler characteristic of a tree.

\begin{rem}
Note that the degree of a  C-region $T$ is the number of times the
boundary of $T$ wraps, under $\sigma_d$, around the analog of the sector of the disk
labeled C on the right side of Figure~\ref{sibportpic}.  A similar
statement holds for R-regions.
\end{rem}

\section{Central Strips}
Let $\mathcal L$ be a sibling $d$-invariant lamination.

\begin{rem}\label{12d} If $\mathcal S$ is  a  full
sibling collection mapping to leaf $\ell=\overline{xy}$, then endpoints $x_i,y_i$ of the preimage leaves alternate
counterclockwise around $\ucirc$:
$x_1<y_1<x_2<y_2<\dots<x_d<y_d<x_1$. (Here we do not  suppose that
$\ell_i=\cl{x_iy_i}$.)  If a leaf is a multiple of
$\frac1{2d}$ long, then it maps to a leaf of length $\frac12$, a diameter. A diameter leaf is either of fixed length or is critical (depending upon
whether $d$ is odd or even).  As these can be handled as special
cases, we consider only full sibling collections not having any leaf
mapping to a diameter.
\end{rem}

\begin{mydef}\label{def-centralstrip}
(Central Strip) Consider the sibling portrait of a full collection
$\mathcal S$ of sibling leaves. Then the {\em central strip} $C$
corresponding to $\mathcal S$  is the closure of the union of all
C-regions $C_i$ with degree at least $2$. The {\em degree} of the
central strip is $deg(C)=\min \{deg(C_i)\}$.
\end{mydef}

A tree with vertices labeled with two colors, and such that no edge connects vertices of the same color, is said to be {\em bicolored}.

 \begin{figure}[h]
    \begin{center}
    \includegraphics[scale=0.25]{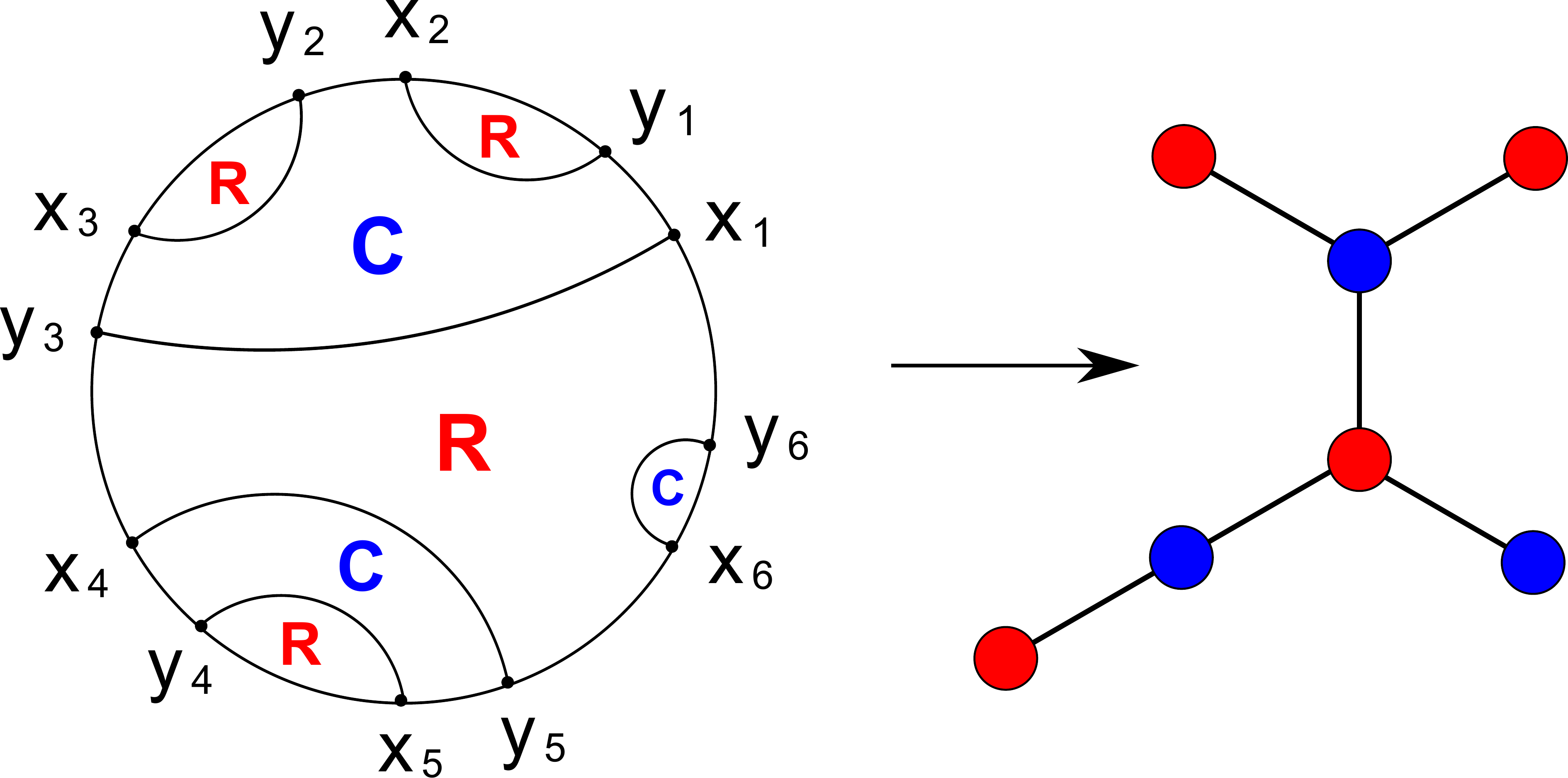}
    \caption{Mapping of a sibling portrait to a bicolored tree.}    \label{treepic}
    \end{center}
  \end{figure}

\begin{thm}[Central Strip Structure] \label{lem-centralstrip}  Let
$\mathcal S$ be a full
sibling collection of leaf $\ell_1$ and its siblings such that $\sigma_d(\ell_1)$ is not a diameter. Then  the following hold.
\begin{enumerate}
\item If some leaf in $\mathcal S$ is of length $>\frac1{2d}$, then  there is a nonempty central strip $C$.
\item The dual graph of the sibling portrait corresponding to 
$\mathcal S$
is  a planar bicolored tree where C-regions are colored one color and R-regions the other.
\end{enumerate}
\end{thm}

\begin{proof}
Suppose ${\mathcal
S}=\{\ell_1,\ell_2,\dots,\ell_d\}$. As in Remark~\ref{12d}, we 
label the endpoints of the sibling leaves $x_1<
y_1<\dots<x_d<y_d<x_1$ in counter-clockwise order so that $x_i$ is
an endpoint of $\ell_i$. Note that the lengths of circular arcs
between successive $x$ and $y$ points alternate between {\em short}
(length $<\frac1{2d}$) and {\em long} (length $>\frac1{2d}$, provided the full sibling collection does not map to a diameter. Without loss of generality assume $(x_i,y_i)$ is short.

To prove (1), assume that some leaf is more than
$\frac1{2d}$ long.  Let $\ell_i$ be a long leaf. We know by
assumption that $\ell_i=\cl{x_iy_j}$, for some $j\not=i$.  We claim that the region $T$ with arc $(x_i,y_i)$ in its boundary is a C-region and the region on the other side of $\ell_i$ is an R-region.  To see this we traverse the region $T$ with arc $(x_i,y_i)$ in its boundary counterclockwise.  Refer to Figure~\ref{sibportpic}.  All sibling leaves map to a single image leaf $\overline{xy}$.  All short arcs map to the shorter (counterclockwise) arc $(x,y)$, and all long arcs map to the longer (counterclockwise) arc $(y,x)$.   Thus, as we move from $x_i$ to $y_i$ in the domain, we traverse the (shorter) arc $(x,y)$ in the range; then as we move from $y_i$ along a leaf emanating from it, we traverse the leaf $\overline{xy}$ from $y$ to $x$ in the range.  Since short and long arcs alternate, we are now again at some $x_k$ in the domain; we cannot be at $x_i$ again, else $\ell_i$ with endpoint $x_i$ is a short leaf; we thus traverse the arc $(x_k,y_k)$ in the domain while we again traverse the arc $(x,y)$ in the range.  Proceeding counterclockwise around the region $T$, we see that we encounter only short arcs in the domain each mapping to the counterclockwise arc $(x,y)$ in the range.  Note that we traversed at least two short arcs in the domain: $(x_i,y_i)$ and $(x_k,y_k)$.  Moreover, we encounter only long leaves in the boundary of $T$.  Thus, $T$ is a C-region with only long leaves and short arcs in its boundary, and $deg(T)\ge 2$.

On the other side of leaf $\ell_i$, by a similar argument, the region $T'$ is an R-region bounded only by long arcs, though it may have both long and short leaves in its boundary.  However, the short leaves in the boundary of $T'$ can bound only degree 1 C-regions.  Thus, $T'$, sharing long boundary leaf $\ell_i$ with $T$, is an R-region.

In the above argument, we have shown both that when some leaf has length $>\frac1{2d}$, the central strip is nonempty, establishing conclusion (1) of the theorem, for we found at least one C-region of degree $\ge 2$, and that any leaf bounds a C-region on one side and an R-region on the other.  So, if a pair of complementary regions share a boundary leaf, then one is a C-region and the other is an R-region.  

To prove (2), we form the dual graph $T_{S}$ of the complementary regions of the sibling collection $\mathcal S$ as follows: let each complementary region  correspond to a vertex of $T_{S}$ and each sibling leaf on the common boundary of two regions correspond to an edge of $T_{S}$.  Since $\cl\disk$ is connected and each sibling leaf disconnects $\cl\disk$, $T_{S}$ is a connected tree.  Refer to Figure~\ref{treepic}.  The proof of part (1) shows that the tree is bicolored, with C-regions being one color and R-regions the other.  If no leaves are of length $>\frac1{2d}$, then all leaves are short and the regions bounded by them are degree 1 C-regions.  The central region bounded by all of them and long arcs of the circle is a degree $d$ R-region.  In this case, there is no central strip.
  This completes part (2) of the proof.
\end{proof}

\begin{prop}
Let  $C_i$ enumerate the complementary components of the full sibling collection $\mathcal S$ meeting the circle in short arcs and $R_i$ enumerate the complementary components of the full sibling collection meeting the circle in long arcs.   Let $T_{S}$ be the corresponding bicolored tree, where  $C'_i$ and $R'_i$ denote vertices corresponding to regions $C_i$ and $R_i$, respectively.  Then 
 the number of edges of $T_{S}= $ $$\sum_i deg(C_i) = \sum_i deg(C'_i) = d = \sum_i deg(R'_i) =\sum_i deg(R_i).$$
\end{prop}

The proof is left to the reader.

\begin{prop} The maximum number of disjoint critical chords that can be contained in a component $T$ of a sibling portrait is $deg(T)-1$.
\end{prop}

\begin{proof} 
Let $T$ be a component of a sibling portrait under $\sigma_d$ of degree $k\le d$.  Without loss of generality, let the (closed) arcs of $T\cap\ucirc$ be $${\mathcal V}=\{[x_1,y_1],[x_2,y_2],\dots,[x_k,y_k]\}.$$ Note that a critical chord must join one component of $T\cap\ucirc$ to another going from a point $x_i+\epsilon_i\in[x_i,y_i]$, $0\le\epsilon_i\le |(x_i,y_i)|$, to a point $x_j+\epsilon_i\in[x_j,y_j]$, $j\not= i$.  Let $\mathcal E$ be a maximal collection of disjoint critical chords in $T$.  It is not hard to see that the cardinality of $\mathcal E$ is at least $k-1$; just draw critical chords from $k-1$ different points of $(x_1,y_1)$ in succession to points, one in each of ${\mathcal V}-\{[x_1,y_1]\}$, in succession in counterclockwise order.

Let $G(T)$ be a graph whose vertices are the elements of $\mathcal V$ and whose edges are elements of $\mathcal E$ joining elements of $\mathcal V$.  
The proposition follows from the following claims about $G(T)$:
\begin{enumerate}
\item  There are no cycles in $G(T)$.
\item  $G(T)$ is connected.
\end{enumerate}
Given the claims, it follows that $G(T)$ is a tree with $k$ vertices, and thus $k-1$ edges.  Hence, $T$ contains at most $k-1$ disjoint critical chords.

To prove (1), suppose, by way of contradiction, that there is a cycle in $G(T)$ of length $m\le k$.  Without loss of generality, assume the cycle includes exactly the first $m$ elements of $\mathcal V$.  Then in $T$ there is a critical chord from a point $x_1+\epsilon_1\in(x_1,y_1)$, $\epsilon_1\ge 0$, to the point $x_2+\epsilon_1\in(x_2,y_2)$. Then there is a critical chord from a point $x_2+\epsilon_2\in(x_2,y_2)$, $\epsilon_2>\epsilon_1$ since the critical chords are disjoint, to the point $x_3+\epsilon_2\in(x_3,y_3)$.  Proceeding in this fashion around the cycle, we finally have a critical chord from a point $x_m+\epsilon_m\in(x_m,y_m)$, $\epsilon_m>\epsilon_{m-1}>\dots>\epsilon_1$, to a point $x_1+\epsilon_m\in(x_1,y_1)$.  But then the first and last critical chords meet in $T$ since $\epsilon_m>\epsilon_1$, a contradiction.

To prove (2), suppose, by way of contradiction, that $G_0$ and $G_1$ are two components of $G(T)$ that are adjacent in counterclockwise order of vertices on $\ucirc$.  Suppose that $(x_r,y_r)$ is the last vertex in $G_0$ and $(x_s,y_s)$ is the last vertex of $G_1$ in counterclockwise order.  Then we can add the critical chord from $y_r$ to $y_s$, connecting $G_0$ to $G_1$, contradicting maximality of $\mathcal E$. It may be necessary to move the endpoints in $\ucirc$ of up to two elements of $\mathcal E$ slightly if they happened to have an $y_r$ or $y_s$ as an endpoint.   
\end{proof}

\subsection{Central Strip Lemma}

The Central Strip Lemma for $\sigma_2$, stated below, was used by
Thurston \cite{Thurston:2009} to show that there could be no {\em wandering triangle} for
a lamination invariant under $\sigma_2$.  A {\em triangle} in a
lamination is a union of three leaves meeting only at endpoints
pairwise and forming a triangle inscribed in $\ucirc$.  A triangle
{\em wanders} if its forward orbit consists only of triangles (i.e.,
no side is ever critical), and no two images of the triangle ever
meet.  This was the first step in Thurston's classification of,  and
description of a parameter space for, quadratic laminations. The Central Strip Lemma for $\sigma_2$ is also used to show that any polygon that returns to itself must return transitive on its vertices; hence, an invariant quadratic lamination cannot have an identity return triangle (see Definition~\ref{id-return}). In a subsequent paper, we will recover and strengthen Kiwi's theorems that
a $d$-invariant lamination cannot have a wandering $(d+1)$-gon, nor an identity return $(d+1)$-gon.

\begin{thm}[Thurston]  Let $C$ be the central strip in a quadratic lamination of leaf $\ell$ with $|\ell|>\frac13$.  Then the following hold:
\begin{enumerate}
 \item The first image $\ell_1=\sigma_2(\ell)$ cannot reenter $C$.
 \item If an iterate
$\ell_j=\sigma_2^j(\ell)$ of $\ell$ reenters $C$, for least $j>1$,
then it must connect the two components of $C\cap\ucirc$.
\end{enumerate}
\end{thm}

 \begin{rem}
$|\ell|=\frac13$ is a special case: (2) holds with $j=1$ and $\ell=\overline{\frac13\frac23}$ maps to itself in reverse order.
 \end{rem}

In order to state and prove a Central Strip Lemma for $d>2$ we will
need to consider the fact that higher degree laminations can have
more than one critical leaf or gap.  To discuss the
distance between chords we use a metric on chords defined by
Childers \cite{Childers:2007} which we call the {\em endpoint metric}.

 \begin{figure}[h]
 \begin{center}
    \includegraphics[scale=0.75]{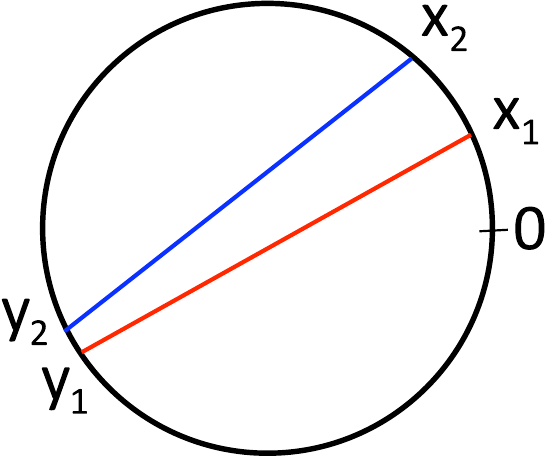}
     \caption{Endpoint distance between two disjoint chords: $\edis(\overline{x_1y_1},\overline{x_2y_2})=|(x_1,x_2)|+|(y_2,y_1)|$.}
       \label{lam-3}
    \end{center}
\end{figure}

\begin{defn}\label{endpoint-metric}
Suppose $\ell_1=\cl{x_1y_1}$ and $\ell_2=\cl{x_2y_2}$ are
chords in $\cl\disk$ meeting at most in one pair of endpoints.  We may suppose the circular order of the
endpoints is $x_1<x_2<y_2\le y_1$.  Define the {\em endpoint distance}
between $\ell_1$ and $\ell_2$ to be
$$\edis(\ell_1,\ell_2)=|(x_1,x_2)|+|(y_2,y_1)|.$$
If $\ell_1=\ell_2$ we define the distance to be $0$.
\end{defn}

See Figure~\ref{lam-3} for an example. We define the endpoint metric only between non-crossing chords.  The reader can check that on such chords it is a metric.
See the proof of Theorem~\ref{lem-centralstrip} for the definition of {\em long} and {\em short} arc length in the following theorem and its proof.

\begin{thm}[Central Strip Lemma] \label{thm-centralstriplemma}
 Let $C$ be a central strip of leaf $\ell$ and
its siblings with  long arc length $>\frac1{d+1}$.  Let $\eta$ be the short arc length.  Then the following hold.
\begin{enumerate}
\item The first image $\ell_1=\sigma_d(\ell)$ cannot reenter $C$.
 \item The second image $\ell_2=\sigma_d^2(\ell)$ cannot reenter $C$ with both endpoints in a single component of $C\cap\ucirc$.
\item If an iterate
$\ell_j=\sigma_d^j(\ell)$ of $\ell$ reenters $C$, for least $j>1$,
and has endpoints lying in one component of $C\cap\ucirc$,
then iterate $\ell_k$, for some $k\le j-1$, gets at least as close in the
endpoint metric as
 $\displaystyle\frac{\eta}{d^{j-k}}$ to a critical chord in $\cl\disk\sm C$.
 \end{enumerate}
\end{thm}

 \begin{figure}[h]
 \begin{center}
    \includegraphics[scale=0.8]{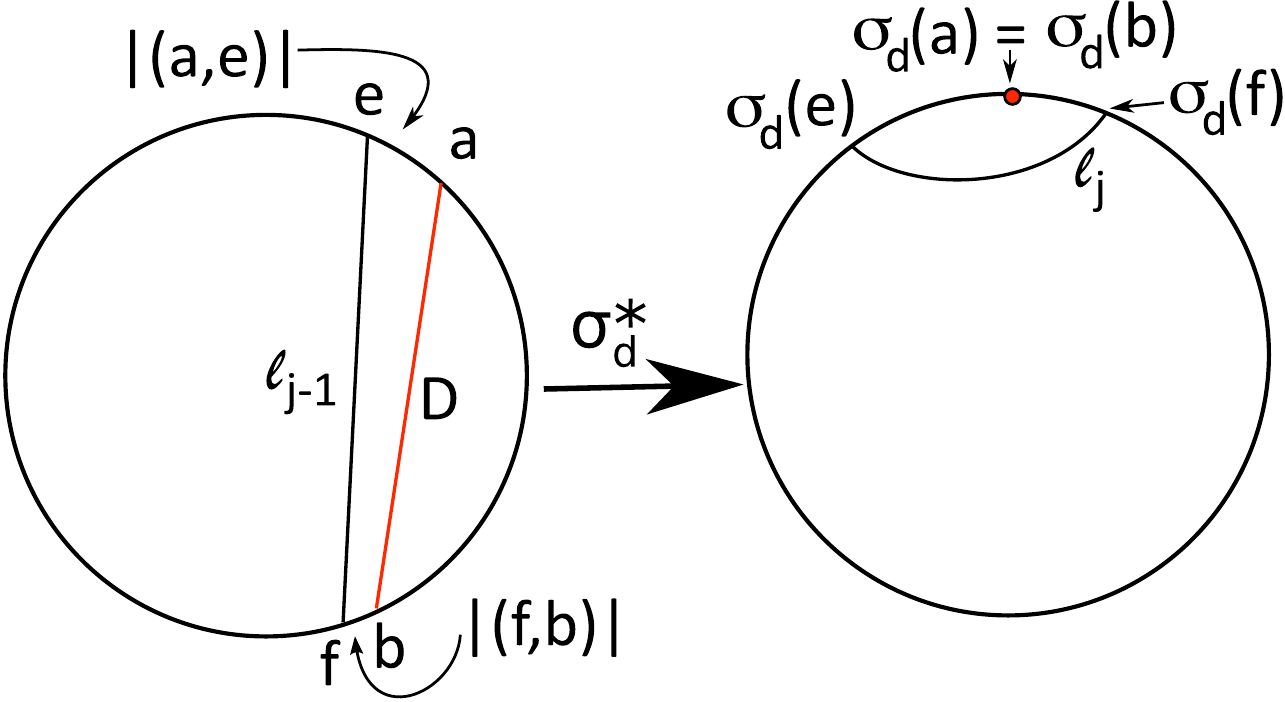}
     \caption{An iterate of $\ell$ approaches a critical chord $D$.}
       \label{CloseToCrit3}
    \end{center}
\end{figure}

\begin{proof}
Suppose $C$ is a central strip with long arc length $>\frac1{d+1}$, so there is leaf $\ell=\cl{xy}$ in its
boundary such that  $|\ell|>\frac1{d+1}$ and no leaf is a multiple of $\frac1{2d}$ long. Let the  short arc length be $\eta$. Then $$\eta<\frac1d-\frac1{d+1}=\frac1{d(d+1)}.$$
 Let
$\{\ell=\ell_0,\ell_1,\ell_2,\dots\}$ be the orbit of $\ell$.  Note
that the length of $\ell_1$ is
$$|\ell_1|=|\sigma_d(\ell_0)|=|(\sigma_d(x),\sigma_d(y)|=d\eta<d(\frac1{d(d+1)})=\frac1{d+1}.$$
Since the length of a component of $C\cap\ucirc$ is $\eta$, the
endpoints of $\ell_1$ cannot lie in one component of $C\cap\ucirc$.  On the other hand, since the length of $\ell_1$ is less than $\frac1{d+1}$, it cannot connect two components, because the long arc length is $>\frac1{d+1}$.  This establishes conclusion (1) of the theorem. 

By Lemma~\ref{lem-length}, the length of $\ell_i$, for $i>1$, will
grow until it is at least $\frac1{d+1}$ long, so will continue not
to fit into one component of $C\cap\ucirc$.  It may be that the
orbit of $\ell$ never reenters $C$, or it may be that it reenters connecting two components of $C\cap\ucirc$.

Suppose now that $\ell_j=\cl{x_jy_j}$ is the first iterate of $\ell$
that reenters $C$, and suppose that the endpoints of $\ell_j$ lie in
one component of $C\cap\ucirc$.  Then $|\ell_j|\le\eta$. The only way
$\ell_j$ can get to  $\eta$ or less in length is by approaching a
critical chord $D=\cl{ab}$ not contained in $C$ sufficiently close
in the endpoint metric. (See Definition~\ref{endpoint-metric} and Figure~\ref{CloseToCrit3}.  In Figure~\ref{CloseToCrit3} the endpoints of $\ell_{j-1}$ are denoted $e$ and $f$.)
Suppose
$$\edis(\ell_{j-1},D)=|(x_{j-1},a)|+|(b,y_{j-1})|\le\frac{\eta}{d}.$$
Then, since $\sigma_d(a)=\sigma_d(b)$, we have
$|\ell_j|=|(x_j,y_j)|=$
$$=|(\sigma_d(x_{j-1}),\sigma_d(a))|+|(\sigma_d(b),\sigma_d(y_{j-1})|$$$$=|(\sigma_d(x_{j-1}),\sigma_d(y_{j-1})|\le d(\frac{\eta}{d})=\eta.$$
  If the last close approach to a
critical chord $D$
 before entering $C$ were at an iterate $k<j-1$, it would have to be even closer (by additional
factors of $\frac 1d$). This establishes part (3) of the theorem.

To see that the first iterate $\ell_j$ of $\ell$ that might enter $C$ with both endpoints in one component of  $C\cap\ucirc$ must have $j>2$, suppose by way of contradiction that  $\ell_2$ has both endpoints in one component.   By the proof of part (3), we have: $$d\eta=|\overline{xy}|<\frac1{d+1}<\frac1d.$$  Now $\ell_1$ must be sufficiently close to a critical chord $D=\overline{ab}$ outside central strip $C$ to shrink $\ell_2$, so that $$ \edis(\overline{xy},D)=|(a,x)|+|(y,b)|\le\frac\eta{d}.$$ But since $|\overline{xy}|<\frac1d$, $\ell_1$ must be under $D$.  
Hence,  $$  \edis(\overline{xy},D)+\overline{xy}=\frac1d.$$  
By the above, we can compute that $$\edis(\overline{xy},D)=\frac{1-d^2\eta}{d}.$$  Our supposition that $\ell_2$ has both endpoints in one component implies that $$\frac{1-d^2\eta}{d}\le\frac{\eta}d.$$  From this and our definition of $\eta$ it follows that $$\frac1{d^2+1}\le\eta<\frac1{d(d+1)},$$ a contradiction for all $d\ge 2$.
This completes the proof of part (2).
\end{proof}


\begin{cor}[Unicritical Central Strip Lemma]\label{unicritCSL}
 Let $C$ be a central strip of degree $d$ of leaf $\ell$ and its siblings for the map $\sigma_d$.  Then no image of $\ell$ can re-enter $C$ with both endpoints in a single component of $C\cap\ucirc$.
\end{cor}

\begin{proof} 
Since $C$ is a central strip of degree $d$ for $\sigma_d$, there must be a full sibling collection ($d$ leaves) in its boundary.  Since it is a central strip, by definition $|\ell|>\frac1{2d}$ as is the length of all its siblings.  But to have room for an all-critical $d$-gon inside the strip, $|\ell|<\frac1d$.  If $|\ell|<\frac1{d+1}$, it will grow in length under future iterates. So we lose no generality in assuming $|\ell|\ge\frac1{d+1}$.  The case where $|\ell|$ is fixed at $\frac1{d+1}$ is trivial. Since there is no critical chord outside $C$, it follows from the Central Strip Lemma~\ref{thm-centralstriplemma}, part (3), that no future image of $\ell$ can re-enter $C$ with both endpoints in a single component of $C\cap\ucirc$.
\end{proof}

\section{Counting Sibling Portraits and Central Strips}

In the proof of Theorem~\ref{lem-centralstrip}, we showed that each sibling portrait corresponds to a bicolored tree.  Now we show that the correspondence is one-to-one up to rotational symmetry.

 \begin{figure}[h]
 \begin{center}
   \includegraphics[scale=0.25]{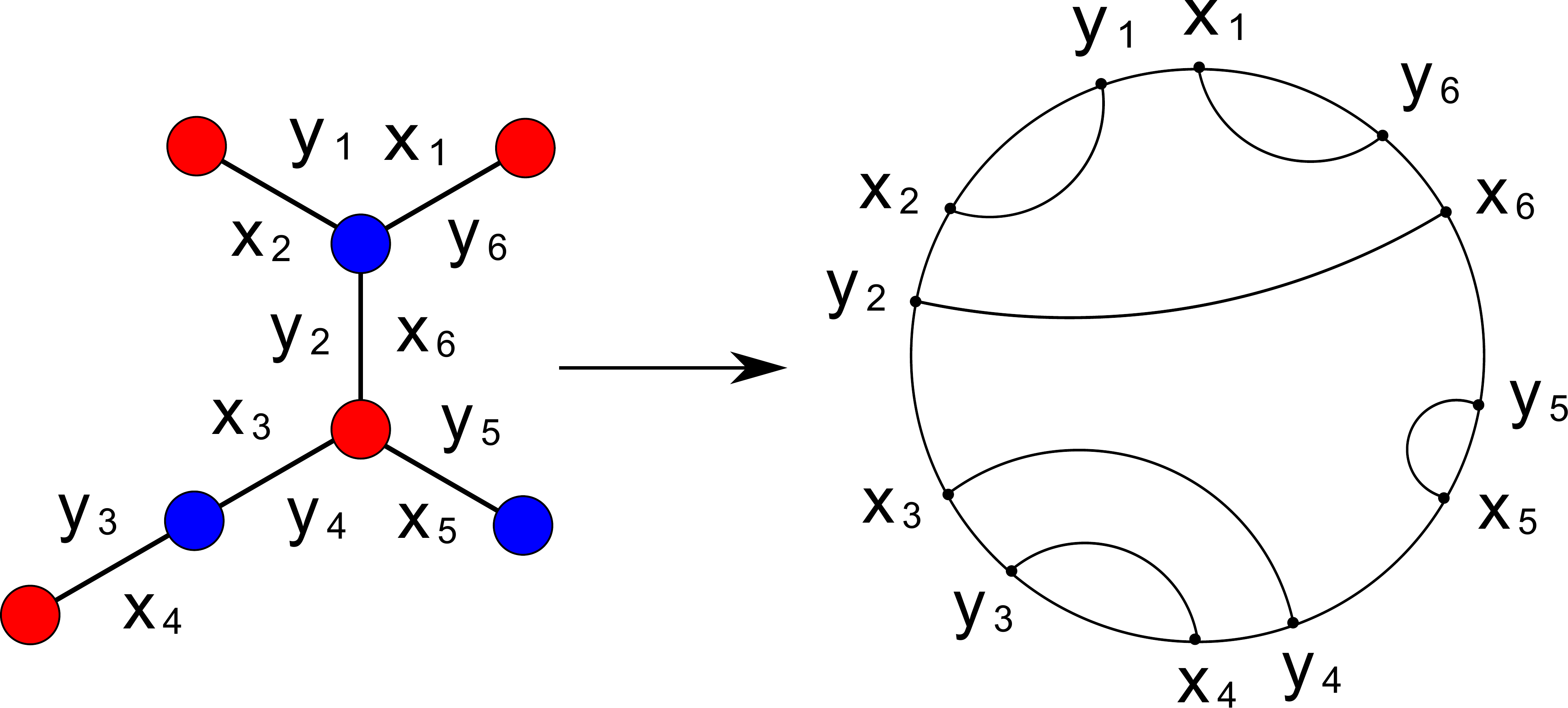}
     \caption{Mapping a bicolored tree to a sibling portrait.}
       \label{treemap}
    \end{center}
\end{figure}

\begin{thm} \label{centraltree}
If $\ell$ is a non-degenerate leaf, not a diameter, then there are $N(d)$ different full sibling collections which map onto $\ell$, distinct up to rotational symmetry, where
$$ N(d)=\frac1d\left(\frac1{d+1}\comb{2d}{d}+\sum_{n|d,n<d}
\phi\left( \frac{d}{n}\right) \comb{2n}{n}\right)$$
and $\phi(x)$ is Euler's {\em totient function}, the number of positive integers less than, and relatively prime to, $x$.
\end{thm}

The proof refers to Figure~\ref{treemap}.

\begin{proof}
The goal is to show that there are just as many different full sibling families (equivalently sibling portraits) which map to the same leaf under $\sigma_d$ as there are different bicolored trees with $d$ edges, up to rotation in the plane. The number of bicolored trees with $d$ edges is known to be $N(d)$ \cite[Th.~21]{bbll:09}. Thus we will use this correspondence to show there are $N(d)$ different sibling portraits mapping to the same leaf. Refer to Fig. ~\ref{treemap} during this proof.

The proof of Theorem ~\ref{lem-centralstrip} illustrates how to map a sibling portrait to a bicolored tree. It is easy to check that if two sibling portraits map to the same bicolored tree (up to rotation) then those two sibling portraits are the same (up to rotation in the plane). Therefore, since for every sibling portrait we can find a unique bicolored tree, there must be at least as many bicolored trees as sibling portraits.

Now assume we are given a bicolored tree with $d$ edges like the one in Fig. ~\ref{treemap}. Since each edge corresponds to a leaf in the full sibling family and each leaf has two endpoints, we may correlate each side of an edge with an endpoint of the leaf. Label the sides of the edges in the tree in a counterclockwise order $x_1, y_1, x_2, y_2, \dots , x_d, y_d$. Since we consider sibling portraits to be the same if one is a rotation of the other, then it does not matter which edge you choose to label $x_1$ with. However, since we generally assume $(x_i, y_i)$ to be a short arc, the vertex of the tree between $x_i$ and $y_i$ must correspond to a $C$-region. Then in the unit circle, connect a leaf between $x_i$ and $y_j$ if they are two sides of the same edge in the tree. This will construct a full sibling family. Similarly, it is easy to check that if two bicolored trees map to the same sibling portrait then they are the same bicolored tree. Therefore for every bicolored tree there is a unique sibling portrait. It then follows that there are exactly as many sibling portraits as bicolored trees.
\end{proof}

The corollary below follows immediately since the only time a sibling portrait does not have a central strip is when all the boundary leaves are short.

\begin{cor}
If $\ell$ is a non-degenerate leaf, not a diameter, then the number of different central strips, distinct up to rotational symmetry, whose boundary leaves map onto $\ell$ is $N(d)-1$.
\end{cor}

\section{Applications to Identity Return Polygons for $\sigma_d$, $d\ge 3$}\label{app}

It will be convenient to be able to refer to points on the circle by their {\em $d$-nary} expansions.  The pre-images under $\sigma_d$ of $0$ partition the circle into $d$ half-open intervals $[\frac{k-1}{d},\frac{k}{d})$, for $1\le k\le d$, labeled successively from $0$ with symbol $\{0,1,\dots,d-1\}$.  A point $x$ of the circle is then labeled with its {\em itinerary}, an infinite sequence $t_0t_1t_2\dots t_n\dots$ of symbols selected from $\{0,1,\dots,d-1\}$, based upon which labeled interval $\sigma_d^n(x)$ lies in.  If a sequence of symbols repeats infinitely a finite sequence of length $n$, we write $\overline{t_0t_1t_2\dots t_n}$ to indicate the infinite sequence.  For example, the point $\frac13$ is of period 2 under $\sigma_2$.  Thus, its $\sigma_2$-itinerary, or binary expansion,  is $\overline{01}$.  However, the $\sigma_3$-itinerary, or ternary expansion, of $\frac13$ is $1\overline{0}$ since $\sigma_3(\frac13)=0$.  Using the $d$-nary expansion of a point $x$, the map $\sigma_d$ is the {\em forgetful shift} since the map sends $$t_0t_1t_2\dots t_n\dots\mapsto t_1t_2t_3\dots t_n\dots\ .$$

\begin{defn}\label{ident-return-poly} \label{id-return}
A (leaf or) polygon $P$ in a $d$-invariant lamination is called an {\em identity return polygon} iff $P$ is
periodic under iteration of $\sigma_d$, the polygons in the orbit of $P$ are pairwise disjoint, and on its first return, each vertex (and thus each side) of $P$ is carried to itself by the identity.
\end{defn}

\begin{figure}[h]
 \begin{center}
   \includegraphics[scale=0.3]{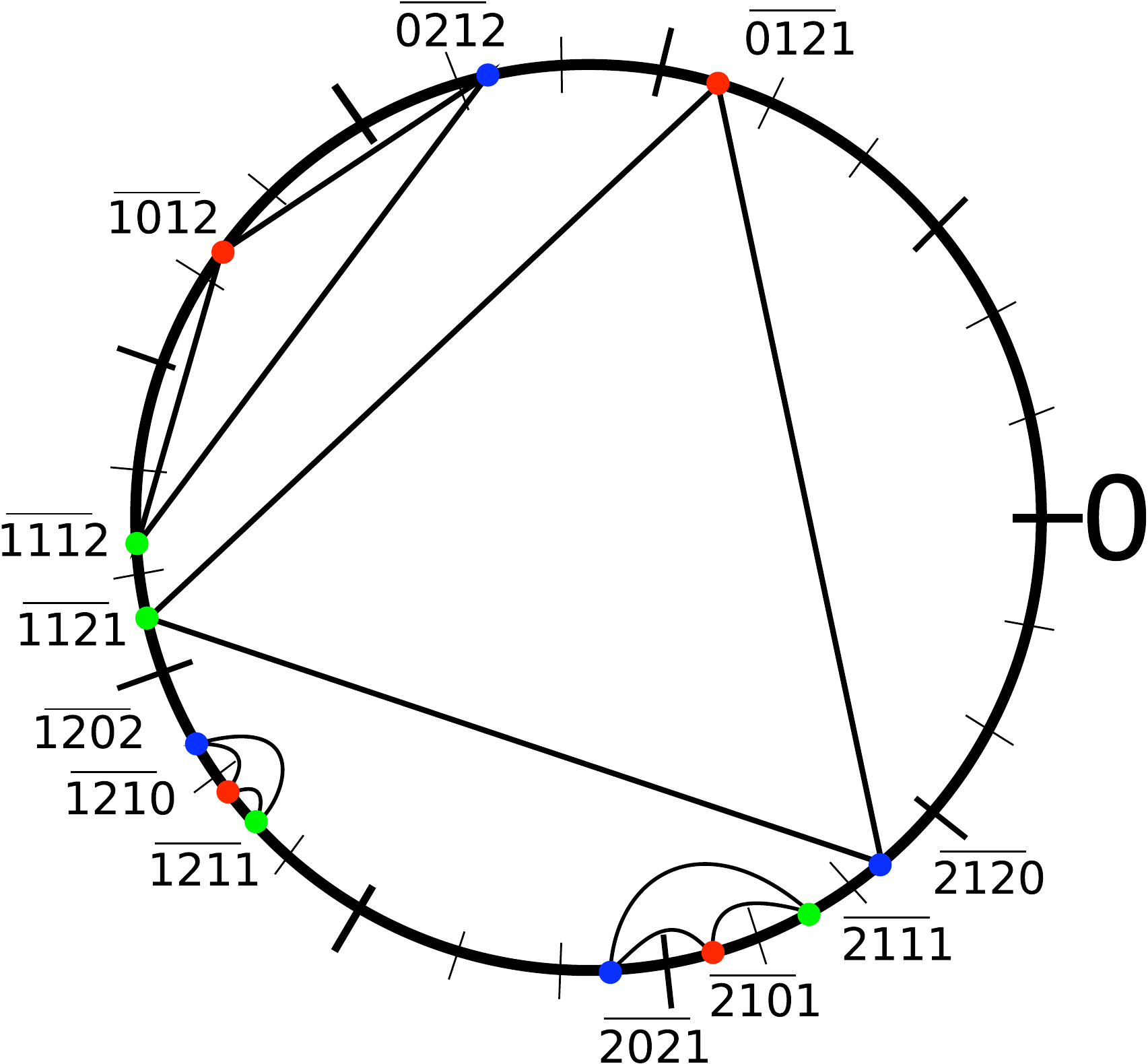}\hspace{0.1in}
    \includegraphics[scale=0.25]{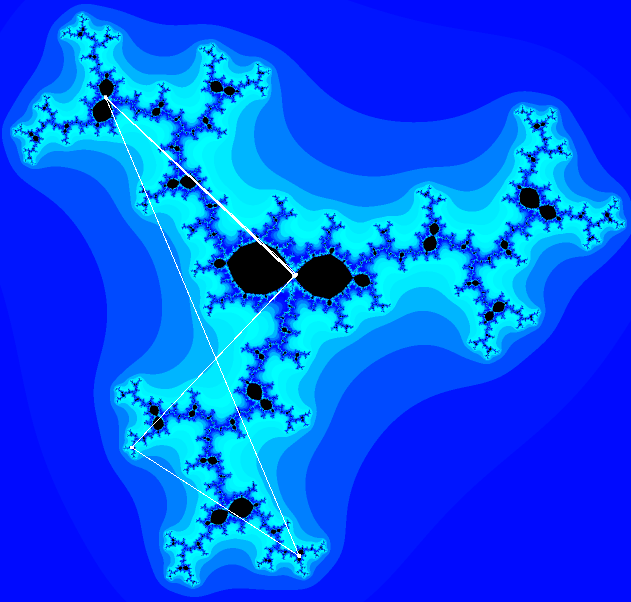}
     \caption{A period 4 identity return triangle for $\sigma_3$ and a corresponding polynomial Julia set with the orbit of the branch point indicated (after Kiwi).}
       \label{Kiwi}
    \end{center}
\end{figure}

Note that we require the identity return polygon $P$ be in a sibling $d$-invariant lamination. This is because without this restriction, one can produce examples of identity return polygons that satisfy the other conditions of Definition~\ref{id-return}, but could not correspond to a Julia set.  (See Figure~\ref{irt2} in connection with the proof of Lemma~\ref{nofixed} for an example.) A periodic polygon in a sibling $d$-invariant lamination corresponds to a periodic branch point in a Julia set.  Locally, the circular order of the branches is preserved by the polynomial.  It is a non-obvious consequence of the definitions that in a sibling $d$-invariant lamination, a polygon maps under $\sigma_d$ preserving the circular order of its vertices \cite[Theorem 3.2]{Mimbs:2013}.    An identity return polygon as we have defined it corresponds to a periodic branch point in a connected Julia set of a polynomial that returns to itself with no rotation around the branch point.  See Figure~\ref{Kiwi} for an example of Kiwi \cite{Kiwi:2002}.

 \begin{figure}[h]
 \begin{center}
   \includegraphics[scale=0.7]{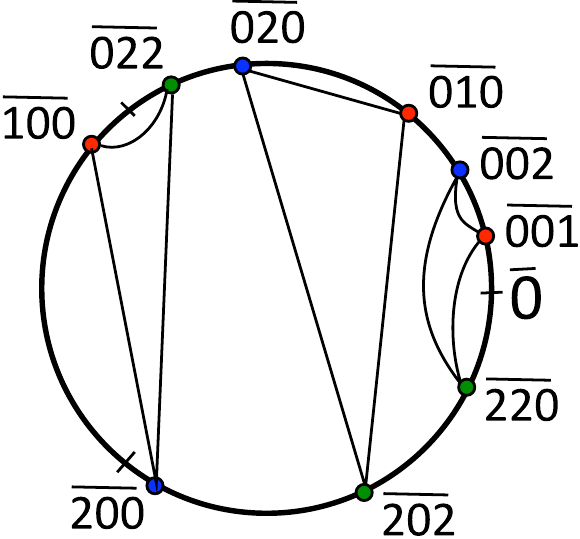}\hspace{0.1in} 
   \includegraphics[scale=0.6]{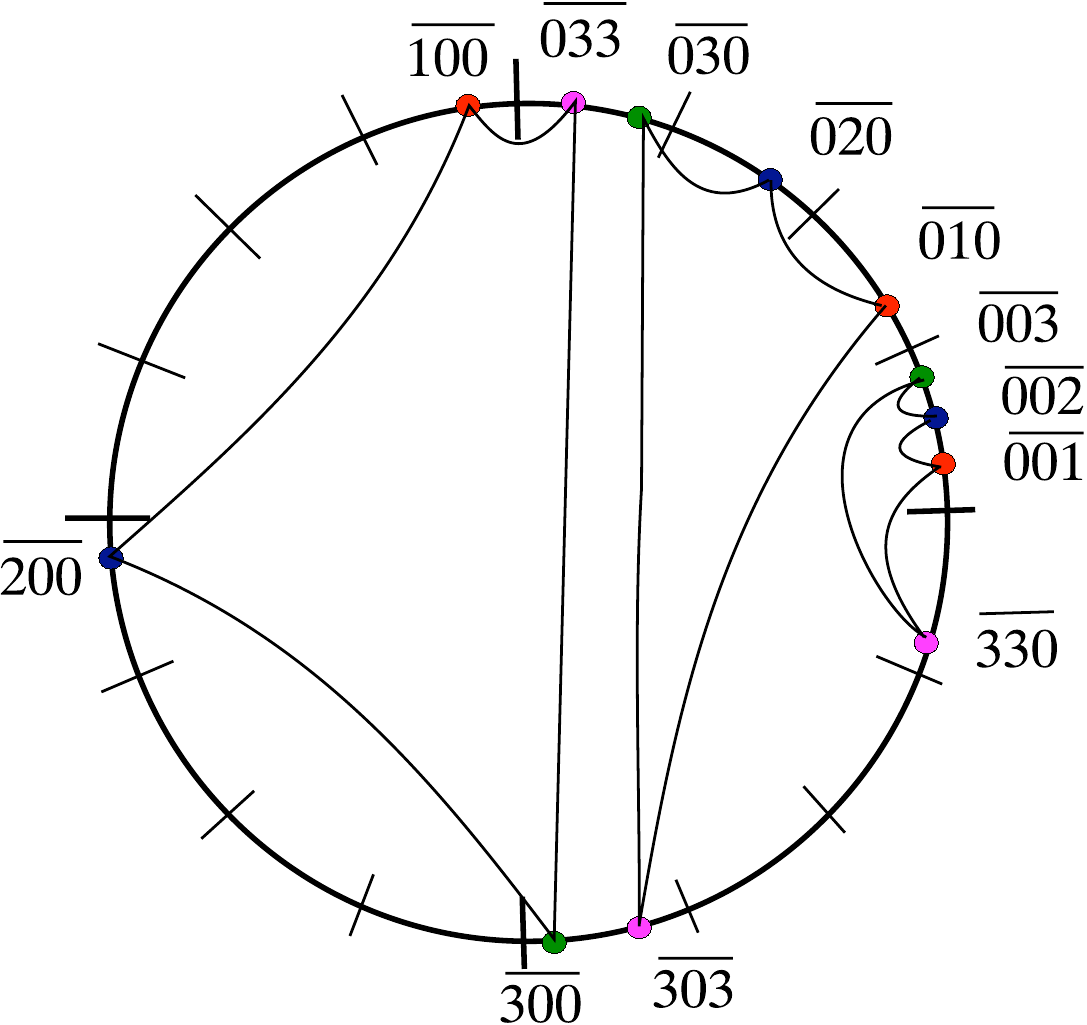}
     \caption{Examples of identity return $d$-gons for $d=3$ and $d=4$ under $\sigma_3$ and $\sigma_4$, respectively; the vertices are labeled by their (periodic) expansions.}
       \label{irt-example}
    \end{center}
\end{figure}

\begin{ex}\label{exp-irp-period3}
There are identity return $d$-gons of period 3 for all $d\ge 3$.  The vertices of an example of such a $d$ -gon for $\sigma_d$ in $d$-nary expansion are $$\{\overline{001}, \overline{002}, \overline{003}, \dots, \overline{00(d-1)}, \overline{(d-1)(d-1)0}\}.$$
\end{ex}

In Figure~\ref{irt-example} we illustrate  an example of an  identity return triangle of period 3 for $\sigma_3$ and  identity return quadrilateral of period 3 for $\sigma_4$.
As shown by Thurston, invariant quadratic laminations cannot have an identity return triangle \cite{Thurston:2009}, though they can have multiple identity return leaves.      Generalizing Thurston's result for quadratic laminations, Kiwi \cite[Theorem 3.1]{Kiwi:2002} proved 

\begin{thm}[Kiwi]\label{kiwi}  A $d$-invariant lamination cannot have an identity return $k$-gon for any $k>d$. \end{thm}

\subsection{Properties of Identity Return Polygons}\label{properties}

In what follows, we extract some facts about identity return polygons in general.  Examples~\ref{exp-irp-period3} and \ref{period-d-1} show that the following theorem is sharp.

\begin{thm} \label{noperiod2} There can be no period 2 identity return $k$-gon for $\sigma_d$ for $k\ge d$.
\end{thm}

\begin{proof}  It suffices to show that there does not exist an identity return $d$-gon of period 2 for $\sigma_d$, since an identity return $k$-gon for $k>d$ would contain an identity return $d$-gon.
We argue by induction on $d$.  It can be easily checked that there does not exist a period $2$ identity return triangle for the map $\sigma_3$; all possible examples result in circular order reversing (see Figure~\ref{irt2}).  Now assume as the induction hypothesis that there does not exist an identity return $(d-1)$-gon of period $2$ under $\sigma_{d-1}$.  By way of contradiction assume that there does exist an identity return $d$-gon of period $2$ under $\sigma_d$.   Let $a_1b_1,a_2b_2,...,a_db_d$ denote the itineraries of vertices of the identity return $d$-gon, with $a_i,b_i \in B=\{0,1,...,d-1\}$ for all $i=1,...,d$. Consider the list of symbols $A=(a_1,...,a_d,b_1,...,b_d)$.  Each element of $B$ must appear in at least one entry of $A$, else we could define a identity return $d$-gon of period $2$ for $\sigma_{d-1}$, which would contain an identity $(d-1)$-gon, contradicting the induction hypothesis. This leads to two cases. 

\smallskip 

\noindent {\bf Case 1. } Suppose that for some symbol $0\leq k\leq d-1$, the symbol $k$ appears more than twice in list $A$. 

\begin{figure}[h]
 \begin{center}
    \includegraphics[scale=0.8]{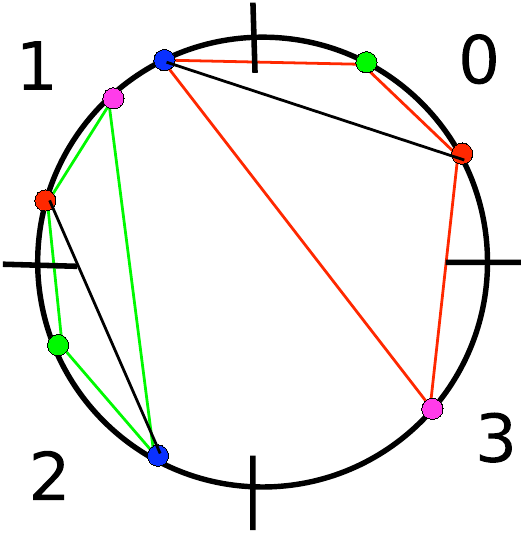}
      \caption{Case 1 for $\sigma_4$ where $m=3$.}
       \label{sigma4-4gon-1}
    \end{center}
 \end{figure}

By the pigeon-hole principle, this means that some other symbol $m\not=k$ will only appear once in $A$. 
Note that any such identity return polygon will be of the form shown in Figure~\ref{sigma4-4gon-1} (illustrated for $d=4$ and $m=3$, up to some rotational symmetry). By removing the vertex with that symbol from the polygon, one obtains a $(d-1)$-gon which uses $d-1$ symbols.  This contradicts the induction hypothesis.

\begin{figure}[h]
 \begin{center}
     \includegraphics[scale=0.28]{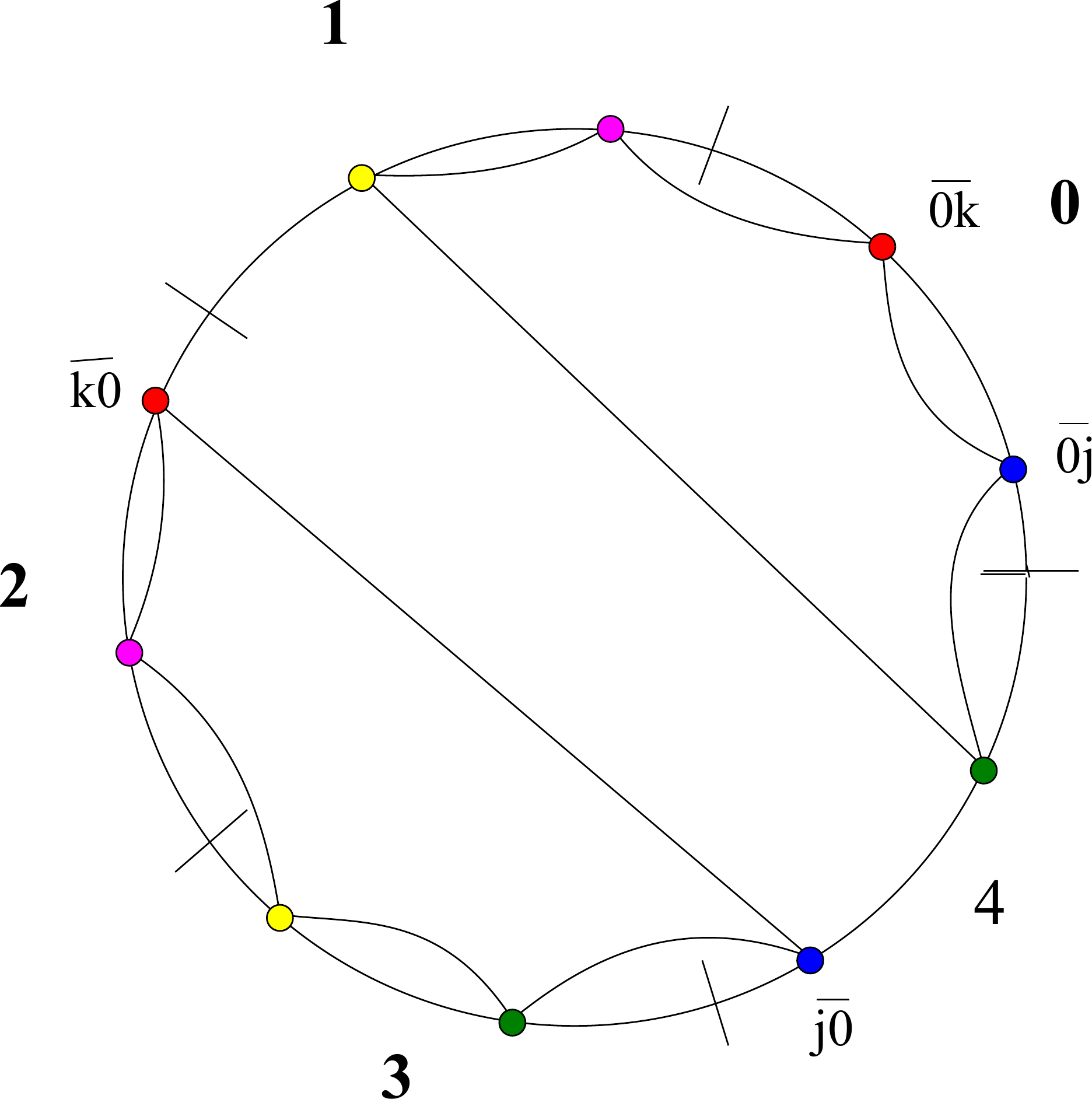} \hspace{0.05in}  
     \includegraphics[scale=0.28]{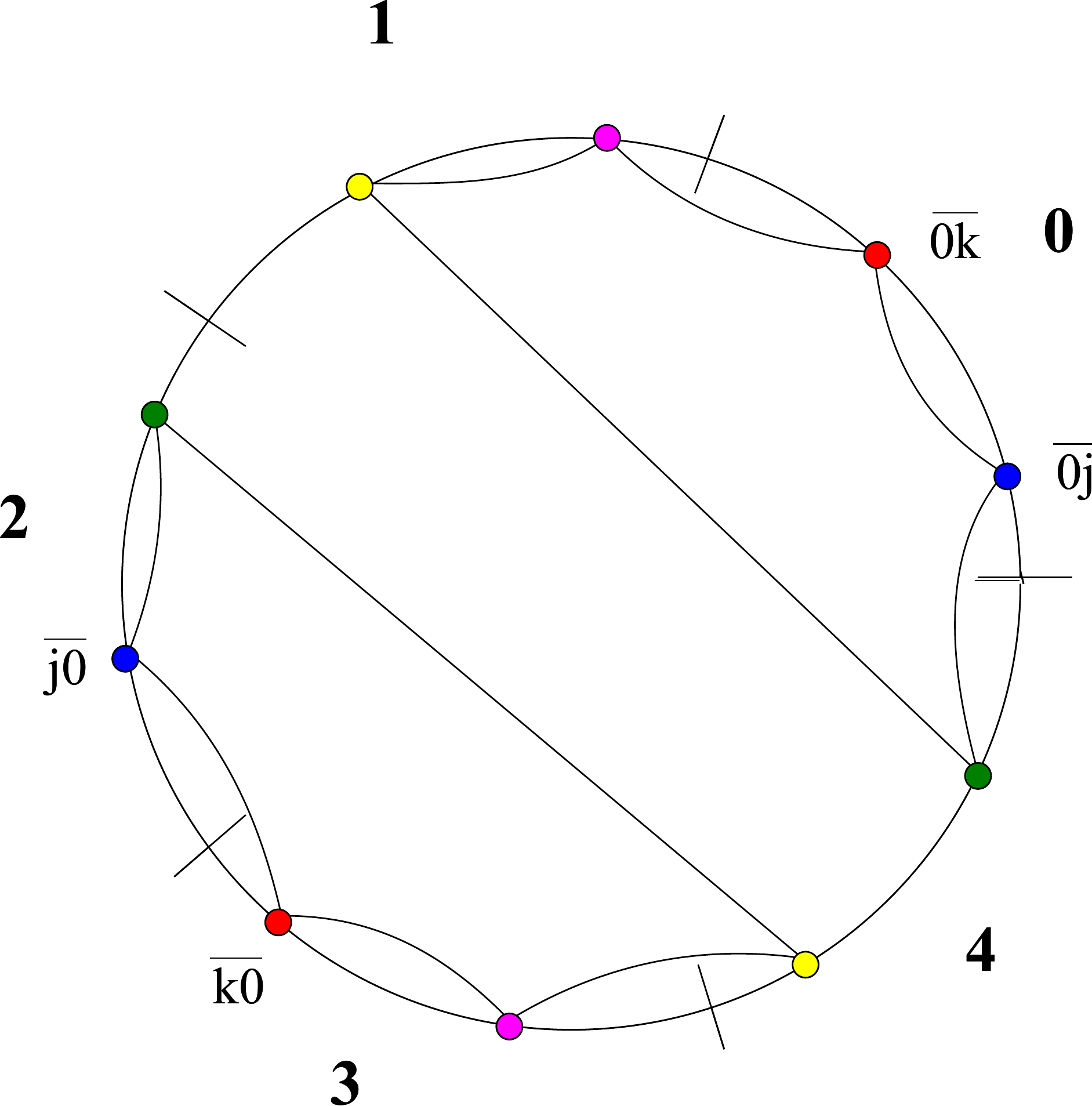}
      \caption{Case 2 for $\sigma_5$ with two vertices of a polygon in the ``$0$'' section.}
       \label{sigma5-5gon-2}
    \end{center}
 \end{figure}

\noindent {\bf Case 2. } Suppose that each symbol $0\leq k\leq d-1$ appears exactly twice in list $A$. 

\smallskip
Then there are two vertices of the polygons in each $1/d$ section of the circle corresponding to the symbols in $B$.  Thus, one of the two  polygons in the orbit must have two vertices in a $1/d$ section of the circle corresponding to the symbol $0$ in $B$ or to the symbol $d-1$ in $B$; otherwise the polygons would cross.   First, let us assume two vertices of a polygon are in the  ``$0$'' section (see Figure~\ref{sigma5-5gon-2}).  The vertex  in this section closest to the point $\overline 0\in\ucirc$, denote it $\overline{0j}$,  represents the first vertex (in counterclockwise circular order from $\overline 0\in\ucirc$)  of the identity return polygon on the circle.   The next vertex in order will be $\overline{0k}$ for some $k>j\in B$. Since the first coordinates of these vertices are $0$, then mapping them forward under $\sigma_d$ will yield vertices $\overline{j0}$ and $\overline{k0}$.  Since $j<k$, these vertices cannot map to the same $1/d$ section.

 Since circular order must be preserved, and all symbols are used, the two vertices must map either to vertices in adjacent $\frac1d$ sections (see Figure~\ref{sigma5-5gon-2} on right), or $\overline{j0}$ must be the largest (in circular order) vertex in the image polygon (see Figure~\ref{sigma5-5gon-2} on left).
Since both vertices have $0$ as their second coordinate,  they must each be the first vertex of the polygons in the sections in which they appear. If the vertices map to adjacent $\frac1d$ sections, this is impossible because $\overline{j0}$ (preceeding $\overline{k0}$  in circular order in the image polygon) would not be the smallest in its section.  If $\overline{j0}$ is the largest vertex in the image polygon, then preserving circular order requires $\overline{k0}$ to be the smallest vertex in the image polygon, contradicting $j<k$.  All options lead to a contradiction, so no such polygon orbit is possible.

The argument is similar if an identity return polygon has two vertices in the $(d-1)$ section.
\end{proof}

\begin{figure}[h]
 \begin{center}
   \includegraphics[scale=1]{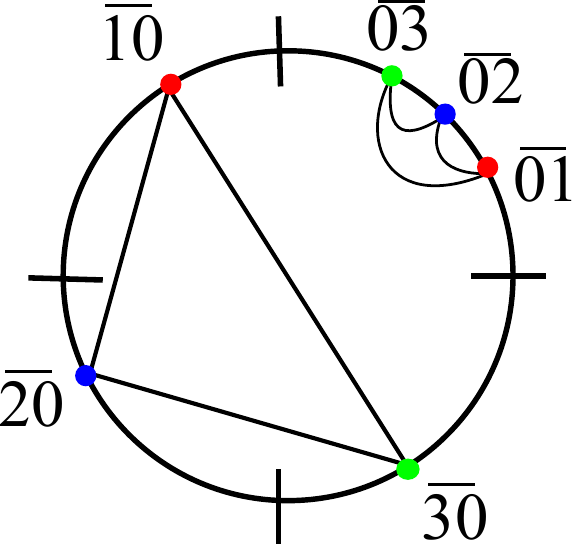}
\caption{Period 2 identity return triangle for $\sigma_4$.}\label{period2-d-1}
    \end{center}
\end{figure}

\begin{ex}\label{period-d-1}
Period $2$ identity return $(d-1)$-gons exist under $\sigma_d$  for all $d>3$. See Figure~\ref{period2-d-1} for $d=4$. The vertices of an example of such a $(d-1)$ -gon for $\sigma_d$ in $d$-nary expansion are $$\{\overline{01}, \overline{02}, \overline{03}, \dots, \overline{0(d-1)}.$$ 
\end{ex}

\subsection{Some Properties of Identity Return Triangles for $\sigma_3$.}

The following propositions present a series of elementary facts about identity return triangles for $\sigma_3$.

 \begin{figure}[h]
 \begin{center}
   \includegraphics[scale=0.85]{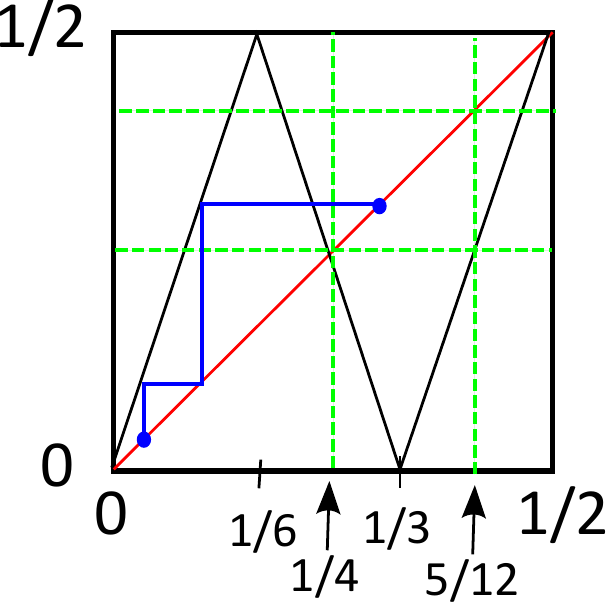} \hspace{0.1in}
   \includegraphics[scale=0.85]{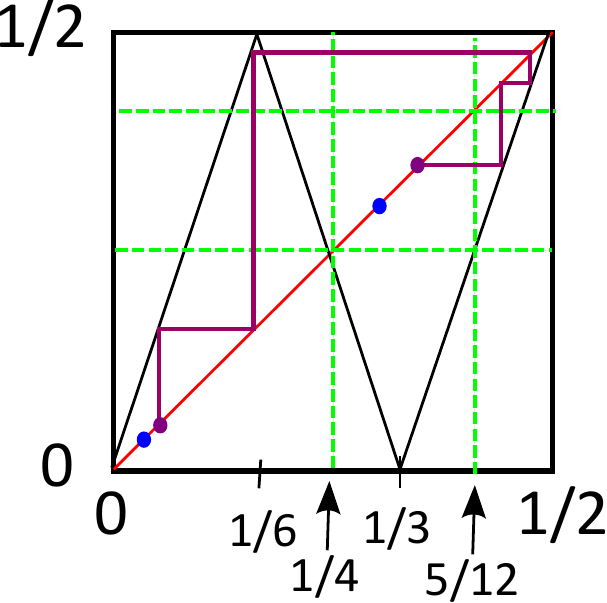}
     \caption{Spider web diagrams for leaf length function $\tau_3$.}
       \label{webdiagram}
    \end{center}
\end{figure}

\begin{prop}\label{irt3-closetocrit}
Let $\ell$ be a leaf in a $3$-invariant lamination  which is not eventually of fixed length.  Then $\ell$ must eventually get within $\frac1{12}$ of a critical length.
\end{prop}

\begin{proof}  Let $\ell$ be a leaf in an invariant lamination under $\sigma_3$ whose length is not eventually fixed.  
By Proposition~\ref{fixedlengths}, the fixed lengths under the leaf length function $\tau_3$ are $\frac14$ and $\frac12$.   By Lemma~\ref{lem-length}, each leaf of length $<\frac14$ must eventually iterate to a leaf $\ell_i=\sigma_3^i(\ell)$ of length $\ge\frac14$ (see ``spider web diagrams''  in Figure~\ref{webdiagram}).  So assume $|\ell|>\frac14$.  If $\frac14<|\ell|<\frac5{12}$, then we can place a critical chord so that in the endpoint metric (Definition~\ref{endpoint-metric}) $\ell$ is within $\frac1{12}$ of the critical chord.  So we may assume $\frac5{12}<|\ell|<\frac12$.     We apply the leaf length function $\tau_3$ iteratively, observing that $\tau_3(|\ell|)$ is repelled from $\frac12$ because the graph of $\tau_3$ is below the identity on the interval $(\frac13,\frac12)$.  Thus, we see that there is a first $k\ge 1$ such that $\frac14<\tau_3^k(|\ell|)<\frac5{12}$ (see Figure~\ref{webdiagram} on right).  Hence, $\sigma^k_3(\ell)$ is within $\frac1{12}$ of a critical chord.
\end{proof}

 \begin{figure}[h]
 \begin{center}
   \includegraphics[scale=1]{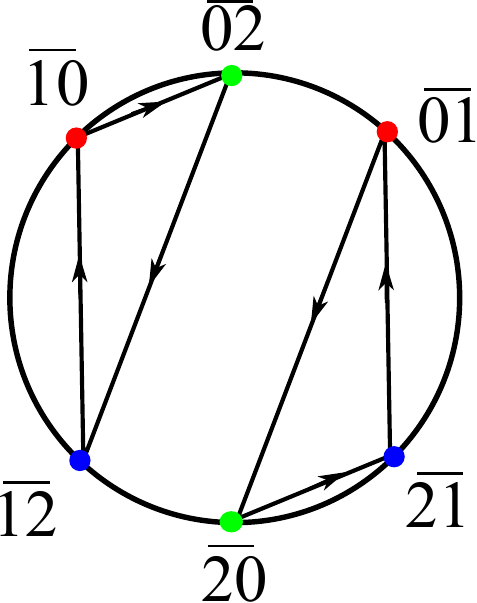}
     \caption{Example of an ``almost'' identity return triangle for $\sigma_3$ of period 2; it fails because the triangle maps forward reversing circular order.}
       \label{irt2}
    \end{center}
\end{figure}

\begin{prop}\label{nofixed}
An identity return triangle in a $3$-invariant lamination cannot have a side of fixed length.
\end{prop}

\begin{proof} 
By Proposition~\ref{fixedlengths}, the fixed lengths under the leaf length function $\tau_3$ are $\frac14$ and $\frac12$.  Let $\ell$ be a side of fixed length of an identity return triangle $T$.  If $|\ell|=\frac12$, then $\ell$ is a diameter, and if $\ell$ is not fixed, $\sigma_3(\ell)$ will cross $\ell$, a contradiction of $\ell$ being a leaf of an invariant lamination.  On the other hand, if $\ell$ is fixed, then $\sigma_3(T)$ meets $T$, contradicting $T$ being an identity return triangle. If $|\ell|=\frac14$, then the period of $\ell$ cannot be greater than 3 or it will meet its iterated image.  If the period of $\ell$ were 3, then both endpoints of $\ell$ would have period 3 repeating ternary expansions (see Figure~\ref{irt-example} for an example). But if we add $\frac14=\overline{01}$ to a repeating ternary expansion for one endpoint of $\ell$, the other endpoint of $\ell$ will not be period 3 repeating in ternary.  It is possible to construct a period 2 triangle which appears to satisfy the definition (see Figure~\ref{irt2}) and has a side of fixed length. However, there is no room for critical chords disjoint from the triangles in the orbit, and the triangle maps forward reversing circular order, thus it cannot be in a sibling $d$-invariant lamination.
\end{proof}

 \begin{figure}[h]
 \begin{center}
   \includegraphics[scale=1]{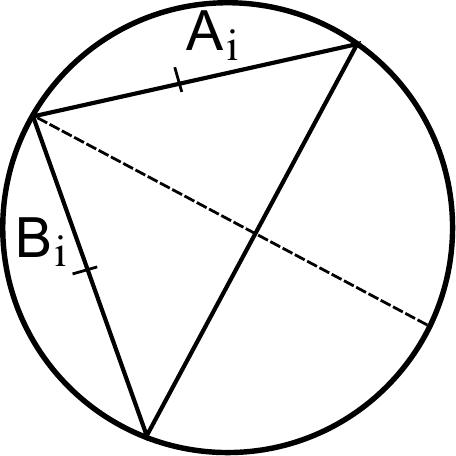} \hspace{0.2in}
   \includegraphics[scale=1]{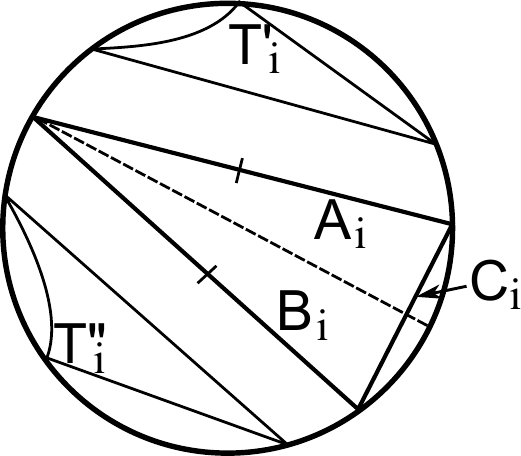}
     \caption{On left, a triangle with two equal sides of length between $\frac14$ and $\frac13$ ; on right, the length is between $\frac13$ and $\frac5{12}$.}
       \label{straddle}
    \end{center}
\end{figure}

\begin{rem}  By way of notation, if $T$ is an identity return triangle of least period $n$ in a 3-invariant lamination, we denote the sides of the triangle by $T=ABC$, and the orbit of the triangle by  $$T=T_0=A_0B_0C_0\mapsto \sigma_3(T_0)=T_1=A_1B_1C_1\mapsto \dots \mapsto T_n=A_nB_nC_n=T_0.$$  This notation can be extended to any identity return polygon.

\end{rem}

\begin{prop}\label{irt3-notequal}
Let $ABC$ be an identity return triangle of period $k$.  If two sides $A$ and $B$ are of the same length, then there exists a unique $i \in [0, k)$ such that $A_i$ and $B_i$ are within $\frac1{12}$ of critical.  Further, there is no critical chord $c$ such that $A_i$ and $B_i$ are both within $\frac{1}{12}$ of $c$ in the endpoint metric.\end{prop}

\begin{proof}
Suppose that $T_0=A_0B_0C_0$ is an identity return triangle and that $|A_0|=|B_0|$. Then for all $i$, we have $|A_i|=|B_i|$.  By Propositions~\ref{irt3-closetocrit} and \ref{nofixed}, there is an $i$ such that $\frac14<|A_i|=|B_i|<\frac5{12}$.   If $|A_i|<\frac13$, as illustrated on the left in Figure~\ref{straddle},  then, similar to the proof of Proposition~\ref{nofixed}, there is no room for two critical chords disjoint from $T_i$, since $A_i$ and $B_i$ together occupy more than half the circle's arc. Hence, there is no room for siblings of $T_i$, so no such sibling invariant lamination.

If $|A_i|>\frac13$, then there are two regions in $\disk\setminus T_i$ where one can place critical chords sufficiently close to $A_i$ and $B_i$, the region with only $A_i$ on its boundary, and the region with only $B_i$ on its boundary.  See Figure~\ref{straddle} on right where we include the siblings $T'_i$ and $T''_i$ of triangle $T_i$.  Since $|A_i|<\frac5{12}$,  each of $A_i$ and $B_i$ is within $\frac1{12}$ of a critical chord. That the $i$th iterate is unique with this property is clear, since any other such iterate would cross $T_i$.
\end{proof}

 \begin{figure}[h]
 \begin{center}
   \includegraphics[scale=0.55]{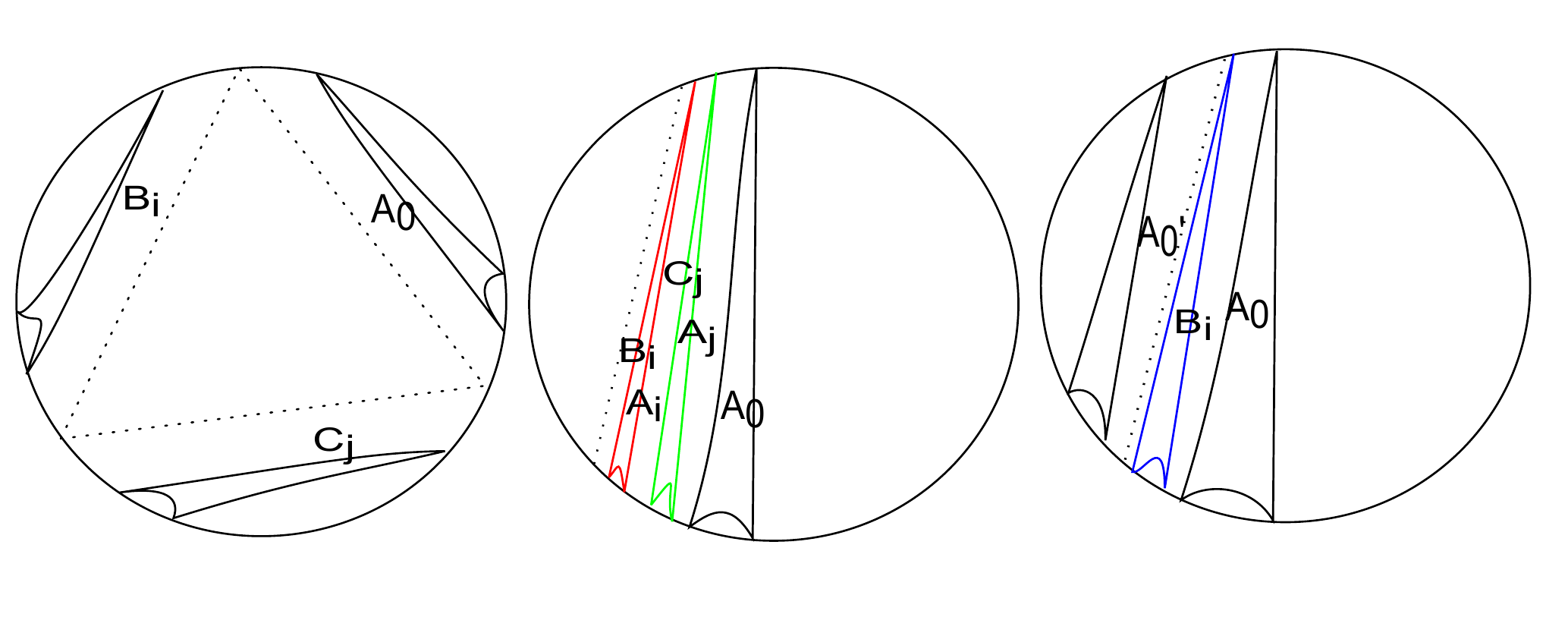} \hfill
     \caption{On left, an identity return polygon approaches three different critical chords; in middle, an identity return polygon approaches just one critical chord; on right, an identity return polygon approaches exactly two critical chords (only one shown).} \label{allcrit}
    \end{center}
\end{figure}

\begin{prop}\label{thinpoly}
Let $T$ be an identity return $n$-gon ($n\ge3$) in a $3$-invariant lamination.  Then in the orbit of $T$ two sides will simultaneously approach within $\frac1{12}$ of a critical length. In fact,  one of the following happens:
\begin{enumerate}
\item Two sides of $T$ approach within $\frac1{12}$ of two different critical chords at the same iterate.
\item Two sides of $T$ approach within $\frac1{12}$ of one side of the same critical chord at the same iterate.
\end{enumerate}
Furthermore, at that iterate these two sides are longer than any other side.
\end{prop}
\begin{rem}
Example~\ref{Kiwi} shows that the first case can occur.  The cubic example in Figure~\ref{irt-example} shows that the second case can occur.
\end{rem}

\begin{proof} Let $T_0=A_0B_0C_0\dots$ be an identity return $n$-gon in a $3$-invariant lamination.  By Lemma ~\ref{irt3-closetocrit} all the sides will eventually get within $\frac1{12}$ of a critical chord. In case (1), the polygon must lie between the two critical chords.  Hence, those two sides are longer than $\frac13$, so all other sides are shorter than $\frac13$.  Now assume case (1) does not occur.

Without loss of generality, suppose that side $A_0$ is within $\frac1{12}$ of a critical chord $\ell$.   So there exists an iterate $i\neq0$ such that side $B_i$ is within $\frac1{12}$ of a critical chord, not necessarily $\ell$, and there exists an iterate $0\neq j\neq i$ such that $C_j$ is within $\frac1{12}$ of a critical chord, not necessarily $\ell$, else we are done.

Suppose that the three sides $A_0$, $B_i$, and $C_j$ are close to the same critical chord $\ell$, and approach no other critical chord closely. Without  loss of generality let $A_0$ be the side furthest away from $\ell$, while still being within $\frac1{12}$ of $\ell$.  See Figure~\ref{allcrit} in middle.  Then $T_i$ is closer to $\ell$ and by the Central Strip Lemma (Theorem~\ref{thm-centralstriplemma}) sides $A_i$ and $B_i$ are long, since side A has not  closely approached a different critical chord before iterate $i$. Triangle $T_j$ cannot be closer to $\ell$ since then, by the Central Strip Lemma, the three sides are simultaneously close to $\frac13$ in length. Therefore, $T_j$ is between $T_i$ and $T_0$ (or siblings). Let $T$ be of period $k$, so $T_0=T_k$. Then, the iterate $T_{k+i}$ is the first time the central strip of $C_j$ is re-entered, and since side $C$ is never close to a different critical chord, $C_{k+i}$ needs to connect two components of the central strip  by the Central Strip Lemma. But $T_{k+i}=T_i$ has $C_i$ inside one component, a contradiction.

Suppose that all three sides are close to three different critical chords at iterates $0$, $i$, and $j$. Then the three critical chords must form an all-critical triangle.  See Figure~\ref{allcrit} on left.   One of the polygons, say $T_0$, must be outermost with respect to the all-critical triangle.  As argued in the previous paragraph,  $T_i$ and $T_j$ are then within the central strip of $A_0$, contradicting in a similar fashion the Central Strip Lemma.   

Therefore, the three sides  $A_0$, $B_i$, and $C_j$ must approach exactly two different critical chords. By the pigeon hole principle, two sides must approach the same critical chord $\ell$. Without loss of generality let those sides be $A_0$ and $B_i$. Let $A_0$ be the one further away from the critical chord $\ell$, but still within $\frac1{12}$ of $\ell$. Suppose that $B_i$ is within $\frac1{12}$ of critical chord at iterate $i$. Then $T_i$ is inside the central strip of $A_0$  and its sibling $A_0^\prime$. Hence,   two sides of the polygon $T_i$ are within $\frac1{12}$ of a critical chord $\ell$ at the same iterate, and all other sides are within a central strip, so shorter.
\end{proof}

 \begin{prop}\label{irt3-longest}
Let $T$ be an identity return $n$-gon in a $3$-invariant lamination. Let $A$ and $B$ be two sides that are within $\frac1{12}$ of a critical chord and $C$ be the longest remaining side. If $A$ and $B$ are simultaneously within $\frac1{12}$ of 
\begin{enumerate}
\item  two different critical chords at the same iterate, or
\item the same critical chord at the same iterate,
\end{enumerate} then the following occur, respectively:
\begin{enumerate}
\item[(a)] $\sigma_3(C)$ becomes the longest side, or
\item[(b)] $\sigma_3(A)$ or $\sigma_3(B)$ becomes the longest side.
\end{enumerate}
\end{prop}

\begin{figure}[h]
 \begin{center}
   \includegraphics[scale=0.5]{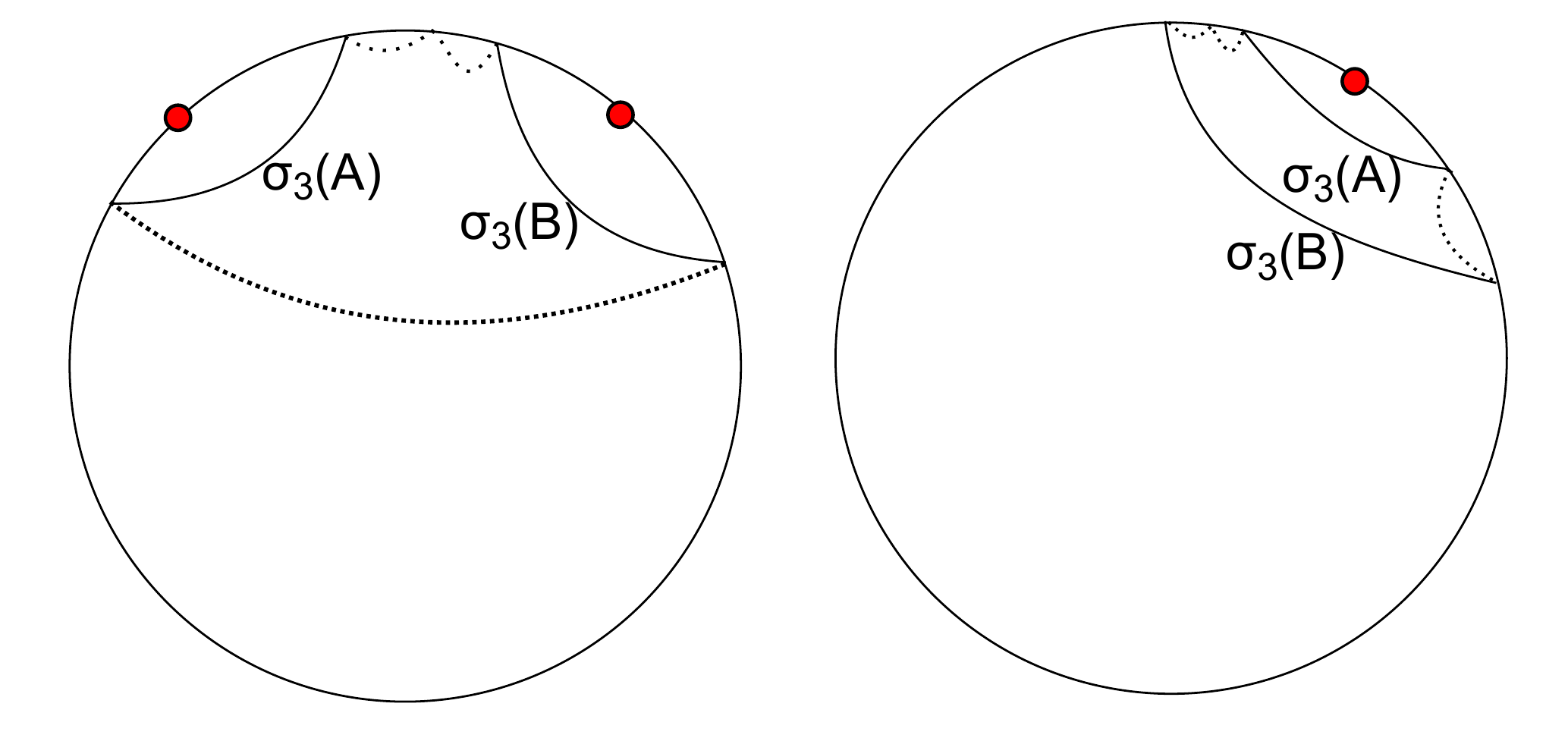}
     \caption{Image of a polygon with critical values marked.}
       \label{critvalue}
    \end{center}
\end{figure}

\begin{proof}

Let $T=ABCD_4...D_n$.

Suppose (1) holds. Then $T$ lies between two critical chords. Since there is a critical chord between $A$ and its sibling $A^\prime$, there must be a critical value underneath $\sigma_3(A)$. Similarly, there is a critical value underneath $\sigma_3(B)$. Furthermore, those two critical values are different since $A$ and $B$ approach two different critical chords. See Figure~\ref{critvalue} on left.

Now, since $T$ lies between two critical chords , $\sigma_3(T)$ lies in one half  of the circle (else it crosses $T$). Since $\sigma_3(T)$ is an $n$-gon, there must be a side $C^\prime$ with $\sigma_3(A)$ and $\sigma_3(B)$ underneath $C^\prime$. Since $\sigma_3(T)$ lies in one half of the circle $C^\prime$ is longer than both $\sigma_3(A)$ and $\sigma_3(B)$.

To see that $C^\prime=\sigma(C)$, consider two cases:

{\bf Case 1}: $|C|\geq\frac16$. Let $|A|=\frac13+\alpha$, $|B|=\frac13+\beta$ and $\gamma=\sum_{k=4}^n |D_k|$
Clearly $\gamma<\frac16$ and $|D_i|<\frac16$ for all $i=4,...n$. Then $|C|=\frac13-\alpha-\beta-\gamma$. 
Then $|\sigma_3(A)|=3|A|-1=3\alpha$, $|\sigma_3(B)|=3|B|-1=3\beta$, 
$|\sigma_3(D_i)=3|D_i|$ for all $i=4,...n$, and $|\sigma_3(C)|=1-3|C|=3\alpha+3\beta+3\gamma=|\sigma_3(A)|+|\sigma_3(B)|+\sum_{k=4}^n |\sigma_3(D_k)|$. Therefore $\sigma_3(C)$ is the longest side.

{\bf Case 2}: $|C|<\frac16$. Then all sides except $A$ and $B$ are shorter than $\frac16$. Thus under the application of $\sigma_3$ the side lengths of $\sigma_3(C),\sigma_3(D_4),...,\sigma_3(D_n)$ triple. Since $C$ was the longest among $C,D_4,...,D_n$, $\sigma_3(C)$ must be the longest among $\sigma_3(C),\sigma_3(D_4),...,\sigma_3(D_n)$. However, we have shown above that there exists at least one side that is longer than $\sigma_3(A)$ and $\sigma_3(B)$. Thus $\sigma_3(C)$ must be the longest.

Cases 1 and 2 establish that (1) implies (a).

Now suppose (2) occurs. See Figure~\ref{critvalue} on right. By the hypothesis, both $A$ and $B$ are within $\frac1{12}$ of a critical chord $\ell$. Then there is the same critical value underneath $\sigma_3(A)$ and $\sigma_3(B)$. Since all remaining sides are between $\sigma_3(A)$ and $\sigma_3(B)$, and since $|\sigma_3(A)|,|\sigma_3(B)|<\frac14$, one of those (the outermost from the critical value) will be the longest.  Thus, we have that (2) implies (b).
\end{proof}

\begin{figure}[h]
 \begin{center}
   \includegraphics[scale=0.75]{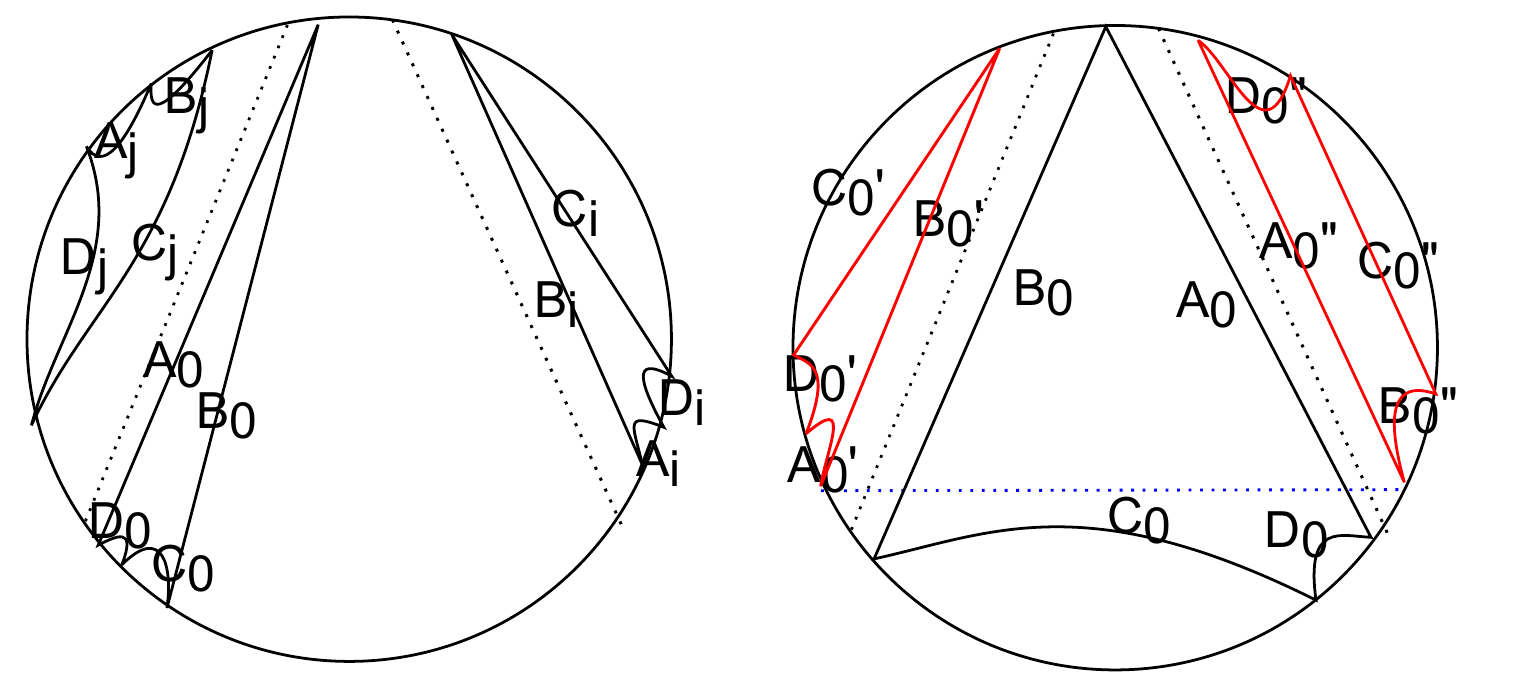}
     \caption{Impossible quadrilaterals under $\sigma_3$:  Case 1 of Theorem~\ref{noIRQinsigma3} on the left and Case 2 on the right.}
       \label{impquad}
    \end{center}
\end{figure}

We conclude this section by proving the special case of Kiwi's Theorem (\ref{kiwi}) for $\sigma_3$.

\begin{thm}\label{noIRQinsigma3}
There are no identity return quadrilaterals under $\sigma_3$.
\end{thm}

\begin{proof}
Suppose that such a quadrilateral $T$, of period $n$ exists. Let $A,B,C,D$ be the sides of $T$, not necessarily in circular order.

{\bf Case 1}:   Suppose that for no iterate $i$ are two sides  within $\frac1{12}$ of two different critical chords. 

Then by ~\ref{thinpoly} there exists an iterate $T_0$ where two sides of $T_0$ are within $\frac1{12}$ of the same critical chord $\ell$. Let $A_0$ and $B_0$ be those sides.  The remaining sides must also eventually approach critical chords within $\frac1{12}$ (Proposition~\ref{irt3-closetocrit}).

{\bf Claim.} We may assume  without loss of generality that every approach to a given critical chord is closer than the previous approach to that critical chord.

To see this, suppose that  at some iterate we have a quadrilateral $T_j$ further away from a given critical chord than quadrilateral $T_i$, $i<j$, while at least one of the sides of $T_j$  is different from the sides of $T_i$ that are within $\frac1{12}$ of a critical chord.   Then, the different ``long'' side of $T_i$ must approach a different critical chord at some iterate before returning to itself.  Otherwise, the Central Strip Lemma is contradicted upon the return to $T_i$. In that case, we could take $T_j$ as a starting point instead of $T_i$, establishing the claim.

By Proposition~\ref{irt3-longest} let $B_1$ be the longest side. Since $|B_1|<\frac14$, so will grow by Lemma~\ref{lem-length}, let $B_i$ be the next approach of side $B$ to a critical chord. If $A_i$ is the second longest side, then we can take this as a starting point. So, without loss of generality, we can suppose that $C_i$ is the second longest. Then by the Claim and the Central Strip Lemma, the two sides must be close to a different critical chord $\ell^\prime$.  Now we have a central strip of sides $A_0,B_0$ around $\ell$ and a central strip of $B_i, C_i$ around $\ell^\prime$. Denote this state of affairs by $[\ell: A_0,B_0]$ and $[\ell^\prime: B_i, C_i]$.  See Figure~\ref{impquad} on left. By Proposition~\ref{irt3-longest} and without loss of generality, we may assume that  $C_{i+1}$ is the longest side.  Let $j$ be the iterate when $C_j$ next approaches a critical chord. 

If $A_j$ is second longest then $T_j$, by the Claim and the Central Strip Lemma, must be close to $\ell$. But this gives us $[\ell:A_j,C_j]$, $[\ell^\prime: B_i,C_i]$, which is up to renaming what we had before. If $B_j$ is second longest then by the Claim and the Central Strip Lemma $T_j$ must be close to $\ell^\prime$. This gives us $[\ell:A_0,B_0]$, $[\ell^\prime: C_j,B_j]$, which is   what we had before but at a later iterate.  If $D_j$ is the second longest, then $T_j$ cannot be within $\frac1{12}$ of either critical chord, without, by the Claim, violating the Central Strip Lemma. Therefore, no such quadrilateral exists.

{\bf Case 2}:  Now suppose that there is an iterate $i$ where two sides, $A_i$ and $B_i$ approach two different critical chords.  

Since this iterate is unique, if we find a subsequent iterate where both sides are within $\frac1{12}$ of the same critical chord, then we can follow the argument in Case 1 to see that no such quadrilateral exists.

Refer to Figure~\ref{impquad} on the right. Let $|A_0|=\frac13+\alpha$, $|B_0|=\frac13+\beta$, and note that $\alpha,\beta<\frac1{12}$. Then $|C_0|+|D_0|=\frac13-\alpha-\beta>\frac16$. First suppose that neither $C_0$ nor $D_0$ is longer than $\frac16$. Then by the Leaf Length Function (\ref{lem-tau}), $|A_1|=3\alpha$, $|B_1|=3\beta$, $|C_1|=3|C_0|$ and $|D_1|=3|D_0|$, and so $|A_1|+|B_1|+|C_1|+|D_1|=1$.  This means that $T_1$ is on both sides of the diameter, so it will intersect $T_0$. 
Therefore, we may suppose that $|C_0|>\frac16$. Then $|A_1|=3\alpha$, $|B_1|=3\beta$, $|C_1|=1-3|C_0|$, and $|D_1|=3|D_0|$. Hence $|C_1|=|A_1|+|B_1|+|D_1|$. 

If $C_0$ is not within $\frac1{12}$ of a critical chord then $C$ grows, so $C_i$ must be within a central strip formed by  $A_0$ or $B_0$ on its first approach to critical length; let it be $B_0$. Then by the Central Strip Lemma, $C_i$ and $B_i$ must be the two longest sides. Now we can follow the argument in Case 1, to see that no such quadrilateral exists.

Therefore, we may assume that $C_0$ is within $\frac1{12}$ of a critical chord.  Since $|C_1|<\frac14$, $C_1$ will continue to grow. So  it must approach a critical length at some iterate $i$. If it approaches a critical chord within the central strip of $A_0$ or $B_0$, we can follow the argument in Case I. So we may suppose that it approaches the same critical chord as at iterate $0$. Then it must go underneath $C_0$, as otherwise side $D$ never gets within $\frac1{12}$ of a critical chord.  There exists an iterate $i$ when side $D_i$ is within $\frac1{12}$ of a critical length. If quadrilateral $T_i$ is inside either of the central strips we can follow the Case I argument. So we may suppose that $D_i$ is within $\frac1{12}$ of the same critical chord as $C_0$. Then by the Central Strip Lemma, sides $C_i$ and $D_i$, must be the two longest sides, so one of them remains the longest. Eventually, one of the sides $A$ or $B$ will become the second longest, otherwise $T$ cannot return to $T_0$. Without loss of generality, let $C$ and $A$ be the two longest sides. Then as they approach critical length at iterate $j$, by the Central Strip Lemma the quadrilateral $T_j$ must be inside the central strip of $A_0$. Again we can follow the argument in Case I.
Therefore there is no such identity return quadrilateral.
\end{proof}

\subsection{Questions}\label{questions}

We conclude with a series of questions and remarks.

\begin{ques}
What is the appropriate generalization of Proposition~\ref{irt3-closetocrit} to $\sigma_d$ for $d>3$?
\end{ques}

In order to be able to use the full power of the Central Strip Lemma, we want leaves of an identity return polygon for $\sigma_d$ to get within $\frac{1}{d(d+1)}$ of a critical chord.  This need not happen even for a $d$-gon for $\sigma_d$, in general.  For example, the period 3 quadrilateral with vertices $\{\overline{132}, \overline{032}, \overline{022}, \overline{200}\}$ under $\sigma_4$ has a side, namely $\overline{\overline{022}~\overline{200}}$, which never gets within $\frac{1}{20}$ of a critical chord.

\begin{ques}  
With reference to Proposition~\ref{nofixed}, can an identity return polygon have a side of fixed length?
\end{ques}

\begin{ques}  
With reference to Proposition~\ref{irt3-notequal}, if an identity return polygon has two sides of the same length, must they simultaneously approach two different critical chords?
\end{ques}

\begin{ques}
What are the appropriate generalizations of Propositions \ref{thinpoly} and \ref{irt3-longest} to $\sigma_d$ for $d>3$?
\end{ques}

We can show that if the two longest sides of an identity return polygon $P$ in a $d$-invariant lamination are simultaneously within $\frac1{2d}$ of the same critical chord, then one of the two longest sides of $\sigma_d(P)$ remains longest.

One can define ``identity return polygon'' without assuming it is in a $d$-invariant lamination.   

\begin{defn}[Alternate Definition of Identity Return Polygon] 
A polygon $P$ in the closed unit disk is called an {\em identity return polygon } iff $P$ is
periodic under iteration of $\sigma_d$, the polygons in the orbit of $P$ are pairwise disjoint, circular order of vertices of $P$ is preserved by the action of $\sigma_d$ on $P$, and on its first return,
each vertex (and thus each side) of $P$ is carried to itself by the identity.
\label{alt-id-return}
\end{defn}

\begin{ques}
Does the existence of an identity return polygon $P$ defined as in Definition~\ref{alt-id-return} imply the existence of a $d$-invariant lamination containing $P$?
\end{ques}

\begin{ques}
What polynomials have Julia sets with a vertex that returns to itself in the pattern of the identity return polygons of Figure~\ref{irt-example}?  Of Example~\ref{exp-irp-period3}, in general?
\end{ques}

\begin{ques}
What is the ``simplest'' $3$-invariant lamination that contains a given identity return triangle?
\end{ques}

Here ``simplest'' might mean, an invariant lamination with no leaves other than preimages of the triangle or their limit leaves, and with each of two sides of the triangle bordering a (different) infinite gap of the lamination.

\begin{ques}
Is there any bound on the number of identity return triangle orbits that a  $3$-invariant lamination can contain?
\end{ques}

It is clear from proofs above that a $3$-invariant lamination can contain only one identity return triangle orbit where two sides of the triangle approach within $\frac1{12}$ of two different critical chords on the same iterate.

\begin{ques}
Given $d> 2$ and a period $p>2$ (necessarily), how many distinct identity return $d$-gon orbits of period $p$ can be formed under $\sigma_d$?
\end{ques}

\begin{ques}
Given $d>2$ orbits of period $p>2$ under $\sigma_d$, is it the case that they form at most one identity return $d$-gon orbit?
\end{ques}

That for $d=2$, the answer is ``yes'' appears to be intimately connected to the detailed structure of parameter space \cite{Thurston:2009}.

\bibliographystyle{annotated}

\end{document}